\pdfoutput=1
\RequirePackage{snapshot}
\documentclass[1p]{elsarticle}
\usepackage[utf8]{inputenc}
\usepackage[T1]{fontenc}

\usepackage{amsfonts,amssymb, amsmath, amsthm,amstext}

\usepackage{bbm} 

\usepackage{verbatim} 

\newtheoremstyle{mytheorem}
{3pt}
{20pt}
{\itshape}
{}
{\bfseries}
{.}
{.5em}
{}%
 
\newtheoremstyle{mydef}
{3pt}
{20pt}
{}
{}
{\bfseries}
{.}
{.5em}
{}%

\theoremstyle{mytheorem}
\newtheorem{theorem}{Theorem}[section]
\theoremstyle{remark}

\usepackage{aliascnt}
\newaliascnt{conj}{theorem}
\newaliascnt{cor}{theorem}
\newaliascnt{lemma}{theorem}
\newaliascnt{prop}{theorem}
\newaliascnt{definition}{theorem}
\newaliascnt{example}{theorem}
\newaliascnt{notation}{theorem}
\newaliascnt{experiment}{theorem}

\theoremstyle{mytheorem}
\newtheorem*{theorem*}{Theorem}

\newtheorem{lemma}[lemma]{Lemma}
\newtheorem{proposition}[prop]{Proposition}

\newtheorem{definition}[definition]{Definition}
\theoremstyle{mydef}

\newtheorem{example}[example]{Example}

\aliascntresetthe{conj}
\aliascntresetthe{cor}
\aliascntresetthe{lemma}
\aliascntresetthe{prop}
\aliascntresetthe{definition}
\aliascntresetthe{example}
\aliascntresetthe{notation}
\aliascntresetthe{experiment}

\usepackage{thmtools}

\usepackage{multicol}
\usepackage{rotating,array}
\usepackage{xcolor}
\usepackage{calc}
\usepackage{booktabs}
\usepackage{mathrsfs}

\usepackage{caption}
\usepackage{subcaption}
\newcommand{\C}{\mathbb C}

\newcommand{\N}{\mathbb N}

\newcommand{\R}{\mathbb R}

\newcommand{\Z}{\mathbb Z}

\renewcommand{\hat}{\widehat}

\newcommand{\indfunc}{\mathbbm{1}}

\newcommand{\hilbert}{\mathcal{H}}

\newcommand{\sign}{\operatorname{sgn}}
\newcommand{\st}{\,:\;}
\newcommand{\lap}{(-\Delta)^{\frac{r}{2}} } 

\newcommand{\inv}{^{-1}}

\def\N{\mathbb{N}}

\def\R{\mathbb{R}}

\def\Fc{\mathscr{F}}
\def\Hc{\hilbert}
\def\Sc{\mathscr{S}}

\def\Pc{\mathscr{P}}
\newcommand{\SmodP}{\mathscr{S}^\prime(\R)/\mathscr{P}}

\def\supp{{\rm supp \,}}
\def\singsupp{{\rm sing \, supp \,} }
\newcommand{\de}[1]{\,\mathrm{d}#1}
\newcommand{\argdot}{\bullet}
\newcommand{\ci}{i}

\newcommand{\jump}{\Theta}

\newcommand{\real}{\operatorname{Re}}
\newcommand{\imag}{\operatorname{Im}}
\newcommand{\p}[1]{\left( #1 \right)}
\newcommand{\ps}[1]{\left\langle #1 \right\rangle}
\newcommand{\abs}[1]{\left| #1 \right|}
\newcommand{\abss}[1]{| #1 |}
\newcommand{\norm}[1]{\| #1 \|}

\newcommand{\complexwvlt}{\kappa}

\usepackage{tikz} 
\usepackage{pgfplots} 
\usetikzlibrary{snakes} 
\usetikzlibrary{patterns}
\usetikzlibrary{matrix} 
\usetikzlibrary{trees}
\usetikzlibrary{decorations.pathmorphing}
\usetikzlibrary{decorations.markings}
\usetikzlibrary{plotmarks}

\usepgfplotslibrary{external}
  \usepackage[pdftex, 
 colorlinks = true,
  linkcolor= blue,
  citecolor = blue]{hyperref}
  
\newlength{\myx} 
\newlength{\myy} 
  
\usepackage{csquotes}

\begin{document}

\pgfkeys{/pgfplots/MyAxisStyle/.style={
        font=\footnotesize,
        tick label style = {font=\scriptsize},
}}

\tikzstyle{plot} = [mark=none, line width=1pt, blue]
\tikzstyle{stem} = [ mark=*, ycomb, line width=0.5pt]
\tikzstyle{vector}=[very thick, ->]
\tikzstyle{versor}=[vector, magenta]
\tikzstyle{stepon}=[only marks, mark = square, blue]
\tikzstyle{stepoff}=[only marks, mark = square*, cyan]
\tikzstyle{mcusp}=[only marks, mark = triangle, red]
\tikzstyle{vcusp}=[only marks, mark = triangle*, orange]
\tikzstyle{signature}=[only marks, mark = square*, orange]
\tikzstyle{signaturevar}=[mark=none, line width=1pt, red]
\tikzstyle{node point}=[inner sep=0pt]
\tikzstyle{sign}=[vector, magenta]

\begin{frontmatter}

\title{Signal Analysis based on Complex Wavelet Signs}
\author{Martin Storath\fnref{epfl}}
\ead{martin.storath@epfl.ch}
\author{Laurent Demaret\fnref{hmgu,tum}}
\ead{laurent.demerat@helmholtz-muenchen.de}
\author{Peter Massopust\fnref{tum}}
\ead{massopust@ma.tum.de}
\fntext[epfl]{Biomedical Imaging Group,
École Polytechnique Fédérale de Lausanne,
1015 Lausanne,
Switzerland}
\fntext[hmgu]{Institute of Computational Biology, Helmholtz Zentrum M\"unchen, Ingolst\"adter Landstrasse 1, 85764 Neuherberg, Germany}
\fntext[tum]{Department of Mathematics, Technische Universität München, Boltzmannstrasse 3, 85747 Garching bei M\"unchen, Germany}
\begin{abstract} We propose a signal analysis tool based on the sign (or the phase) of complex wavelet coefficients, 
which we call a signature.
The signature is defined as the fine-scale limit of the signs of a signal's complex wavelet coefficients.
 We show that the signature equals zero at sufficiently regular points of a signal
 whereas  at salient features, such as jumps or cusps, 
 it is non-zero.
 At such feature points, the orientation of the signature in the complex plane
  can be interpreted as an indicator of local symmetry and antisymmetry.
 We establish that the signature 
 rotates in the complex plane under fractional Hilbert  transforms.
 We show that certain random signals, such as white Gaussian noise and Brownian motions, 
 have a vanishing signature.
 We derive an appropriate discretization and show the applicability to signal analysis.
\end{abstract}
\end{frontmatter}
\section{Introduction}

The determination and discrimination of salient features, such as jumps or cusps, is a frequently occurring task in signal and image processing \cite{morrone1987feature, lindeberg1998feature,  kovesi1999image,jacob2004design, mallat2009wavelet}. 
In many classical approaches, it is assumed that the interesting features of a signal are 
located at the points of low regularity.
In this context, local regularity is measured  in terms of the (fractional) differentiability order,
e.g., in the sense of local Hölder, Sobolev or Besov regularity \cite{mallat2009wavelet,ehler2009nonlinear, abry2009scaling,hoermander2007analysis}.
Since such measures of smoothness only rely on the {modulus} of wavelet coefficients, 
 they do not take into account sign (or phase) information.
However, classical results  in Fourier and wavelet analysis suggest
that  the signs of the wavelet coefficients contain rich information about a signal's structure.
Logan  \cite{logan1977information} has shown  that bandpass signals are characterized, up to a constant, by their zero crossings, which in turn are determined by the sign changes.
The well known results  by Oppenheim and Lim \cite{oppenheim1981importance} indicate that sign information is 
even more important for the reconstruction of images than the modulus. 
If the  amplitude of the Fourier transform of an image  is combined with the sign of
the Fourier transform of another image  then the resulting reconstruction
 is structurally much closer to the second image, whereas the structure of the first one is hardly visible.
Similar observations  have been made for complex wavelet coefficients  \cite{held2010steerable,storath2013amplitude}.
It has been shown by Jaffard \cite{Jaffard:2004,Jaffard:2005}  that
a bounded signal becomes unbounded almost surely if the signs of its wavelet coefficients are randomly perturbed.
This suggests that  perturbations  of the wavelets signs can significantly alter the shape of a signal.
But since the perturbation only affects the sign information 
this change is not taken into account by a purely modulus-based signal analysis.
 Let us illustrate this by an example of Meyer \cite{meyer1992wavelets}. 
Consider the functions $f(x) = \sign x$ and $ g(x) = 2 \log | x|.$
Since $f$ and $g$ are related by the Hilbert transform, their wavelet coefficients are equal with respect to their order of magnitude.
Although the singularities of $f$ and $g$ are structurally very different, they are not delineated by a purely modulus-based signal analysis.

Despite their high information content,
the signs of the wavelet coefficients have received less attention  than the moduli in signal and image analysis.
First indications for the usability of wavelet signs in signal analysis have been given by Kronland-Martinet, Morlet, and Grossmann \cite{kronland1987analysis}. They observed that the lines of constant sign in the wavelet domain converge towards the singularities.
 Reconstruction algorithms from sign/phase information have been presented in 
\cite{mallat1991zero} for wavelet coefficients  and \cite{urieli1998optimal} for the short time Fourier transform.
The phase of the short time Fourier transform  has been recently analyzed in \cite{jaillet2009structure, balazs2011phase}.
In \cite{morrone1987feature, venkatesh1990classification}, 
phase congruency was introduced for signal analysis based on the sign information of Fourier coefficients.
Kovesi \cite{kovesi1999image} added the multiscale aspect to phase congruency using log-Gabor wavelets.
Phase congruency is especially useful for contrast invariant edge detection \cite{kovesi1999image}. 
It has also been shown that phase congruency provides valuable information about local symmetries of a feature point, e.g. symmetric cusps or antisymmetric steps \cite{morrone1987feature, venkatesh1990classification,kovesi1999image}.
Similar observations have been made by Holschneider for the fine scale behavior of the wavelet phase \cite[Ch. 4.3]{holschneider1995wavelets}.
On the flipside, phase congruency lacks so far a rigorous foundation.

In this paper, we propose a new approach to signal analysis based on the sign (or phase) of complex wavelet coefficients. 
Our analyzing wavelets belong to the class of complex wavelets;
this means that their real part and their imaginary part are Hilbert transforms of each other.
Such analyzing functions can be traced back to the seminal work of Gabor \cite{gabor1946theory}.
They have been applied successfully to many signal and image processing applications \cite{kingsbury1999image,kingsbury2001complex, forster2004complex, selesnick2005dual, tu2005analysis}.
Here, we investigate the fine scale behavior of the signs of complex wavelet coefficients.
More precisely, we consider the complex-valued quantity 
\begin{align*}
	\sigma f (b) = \lim_{a \to 0+} \sign\ps{f, \kappa_{a, b}} = \lim_{a \to 0+} \frac{\ps{f, \kappa_{a, b}}}{|\ps{f, \kappa_{a, b}}|},
\end{align*}
where $f$ is a real-valued signal, $\kappa$ is a complex wavelet, $a > 0$ denotes the scale, and $b \in \R$ the location. 
If the limit does not exist, we set $\sigma f(b)$ to $0.$
We call the  quantity $\sigma f(b)$ 
the \emph{signature of $f$ at location $b.$}
The basic idea is that a nonzero signature indicates the presence a salient point of the signal $f,$
e.g. a step or a cusp, whereas at regular points, the signature equals zero.
We first show that the signature is equal to zero if the signal $f$ is locally polynomial around $b.$ 
We then establish that the signature is $\pm 1$ for cusp singularities of power type
whereas it is equal to $\pm i$ for step singularities.
Thus, the orientation of the signature within the complex plane may be interpreted as an indicator of local symmetry or antisymmetry. 
We further show that 
the singular support, which consists of all points where the signal is not locally $C^\infty$,  
does in general not coincide with the  support of  signature.
Thus, a singularity in the classical sense need not coincide with a singularity in the sense of  signature.
We further show that Gaussian white  noise and fractional Brownian motion 
contain no salient points in the sense of the signature.
We propose a suitable discretization and validate the theoretically developed concepts
 by numerical experiments. Finally, we provide connections of our discretization to phase congruency.

\subsection{Organization of this article}

The structure of this article is as follows.
In \autoref{sec:signature} we introduce the signature and present some of its basic properties. 
We compute the signature of regular points of a signal in \autoref{sec:signatureRegular}
and of singular points such as jumps and cusps in \autoref{sec:signatureSingular}.
We investigate the behavior of signature under the fractional Hilbert transform in 
\autoref{sec:signatureHilbert}. 
In \autoref{sec:signatureRandom} we show that certain random signals have a vanishing signature.
In \autoref{sec:signatureGeometric}, 
we establish relations to the singular support and give a geometric interpretation.
We propose  a discretization and illustrate our theoretical results by numerical experiments in \autoref{sec:numeric}.

\section{Definition of signature and its basic properties } \label{sec:signature}

In this section, we introduce the concept of signature and state its basic properties.
We start with some preliminary notations. We define the Fourier transform of a Schwartz function $f \in \mathscr{S}(\R)$ by
\[
   \mathscr{F} (f)(\omega) = \widehat{f} (\omega) = \int_\R e^{-i\omega x} f(x) dx.
\]
Likewise, we use the above notation for the usual extension to the space of tempered distributions $\mathscr{S}^\prime(\R)$. 
Furthermore, $\mathscr{F}^{-1}(f)$ and $f^\vee$ denote the corresponding inverse Fourier transform of $f$.
We define the operators of translation by $r\in \R,$  $T_r,$ and dilation by $\nu \in \R \setminus \{0\}$,
$D_\nu$, by
\[
 T_r f(x) = f(x-r) \quad \text{and}\quad
	  D_\nu f(x) = \frac{1}{\sqrt{\nu}}\, f\p{\frac{x}{\nu}}, \\
\]
respectively. 

\subsection{Definition of the signature}

Our analysis will rely on wavelets with a one-sided, compactly supported, non-negative frequency spectrum. For brevity, we call this class signature wavelets.
\begin{definition}\label{def:signature_wavelet}
We call a non-zero function $\kappa \in \Sc({\R}, \C)$ a
	 \emph{signature wavelet} if 
$\kappa$ has a one-sided, compactly supported, and non-negative Fourier transform, i.e.,
	  \begin{align*}
	  \supp \hat\kappa \subseteq [c,d], \quad 0<c<d<\infty, \quad\text{and}\quad
	  	   	\hat\kappa(\omega) \geq 0, \quad \text{for all }\omega \in \R.
	  \end{align*}
\end{definition}

Since their Fourier transforms  vanish around zero,
signature wavelets have infinitely many vanishing moments.
The class of signature wavelets is invariant under convolution operators 
whose kernels have a positive Fourier transform. In particular, 
the class is invariant under fractional Laplacians $\lap$ which are defined by 
 \begin{align}\label{eq:frac_lap}
     \Fc\{ \lap f\} = |\argdot|^r \, \widehat{f}, \quad \mbox{for } r \in \R.  
 \end{align} 
The fractional Laplacians
 are well defined for all $r \in \R$ on the tempered distributions modulo the polynomials $\Pc$, which we denote by $\Sc'/ \Pc.$ 

The real part and the imaginary part of a signature wavelet
are Hilbert transforms of each other (up to a factor $-1$).
Complex-valued functions with this property are often called analytic signals.
 The wavelet system associated with a signature wavelet $\kappa$ is defined as the family of functions
\[
   \kappa_{a,b}(x) = \frac{1}{\sqrt{a}}\, \kappa \left( \frac{x-b}{a} \right),
   \qquad \text{where }a  > 0 \text{ and } b \in \R.
\]
An example of a signature wavelet is given by the inverse Fourier transform of the (one-sided) Meyer window $W,$  i.e.,
\begin{align}\label{eq:meyer_signature_wavelet}
 	\kappa(x) = \Fc \inv (W)(x),
\end{align}
where $W$ is given by
\begin{align}\label{eq:meyer_wavelet}
	W(\omega) = 
  \begin{cases}
    \cos\left(\frac{\pi}{2} \,g(5 - 6\omega)\right), &\text{for } \frac{2}{3} \leq \omega < \frac{5}{6}, \\
    1 , &\text{for } \frac{5}{6} \leq \omega < \frac{4}{3}, \\
    \cos\left(\frac{\pi}{2}\, g(3\omega - 4)\right), &\text{for } \frac{4}{3} \leq \omega < \frac{5}{3}, \\
    0, &\text{else,}
  \end{cases}
\end{align}
with a smooth function $g$ satisfying $ g(\xi)=  0$ for   $\xi \leq 0,$ 
$ g(\xi)=  1$ for   $\xi \geq 1,$ and  $g(\xi) + g(1 - \xi ) = 1$ otherwise.
Particular choices for $g$ are given in \cite{daubechies1992ten}. 
 The graph of such a $\kappa$ is depicted in \autoref{fig:meyer}. We employ this wavelet in the numerical experiments in \autoref{sec:numeric}.

	\begin{figure}
	\centering
	\def\thisfigwidth{0.55\textwidth}
  \def\thisfigheight{0.3\textwidth}
    \def\thisfigheightTwo{0.4\textwidth}
  \def\dataFolder{../plots}
     \begin{tikzpicture}
    \begin{axis}[
      footnotesize,
      width=\thisfigwidth,
      height=\thisfigheightTwo,
      axis y line=center,
      axis x line=center,
      xlabel=$x$,
      scaled ticks=false,
      ytick = {-0.005,0.005},
      xtick = {-4,  4},
         tick label style={font=\footnotesize},    
]
            \addplot[plot, blue] table[x =x,y =fRe]
      {\dataFolder/functions/meyer.table};  
      \addlegendentry{$\real \kappa(x)$}
      \addplot[plot, red] table[x =x,y =fIm]
      {\dataFolder/functions/meyer.table};  
      \addlegendentry{$\imag \kappa(x)$}
    \end{axis}
  \end{tikzpicture}
\hfill
     \begin{tikzpicture}
    \begin{axis}[
      footnotesize,
      width=\thisfigwidth,
      height=\thisfigheight,
      axis y line=center,
      axis x line=center,
      xlabel=$\xi$,
      ylabel=$\hat\kappa(\xi)$,
      ymax=1.1,
            ymin=-0.1,
      xtick = {-2, -1, 1, 2},
      legend style={at={(0.03,0.97)},anchor=north west},
      tick label style={font=\footnotesize},    
]
            \addplot [plot] table[x =xi,y =fHat]
      {\dataFolder/functions/meyer.table};  
    \end{axis}
  \end{tikzpicture}
  \caption{A Meyer-type signature wavelet $\kappa$ (left) and its Fourier transform $\hat\kappa$ (right). }
  \label{fig:meyer}
	\end{figure}
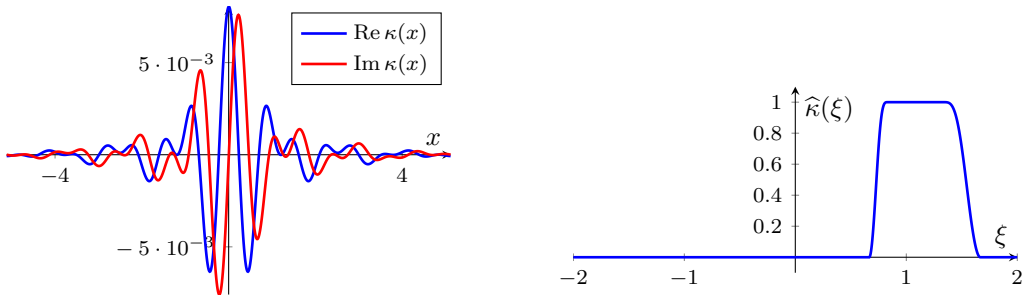

The quantity we are interested in  is the fine scale limit
of the signs of the complex wavelet coefficients. As it involves the signs of the wavelets coefficients we call this quantity signature. Recall that the sign of a complex number is defined by
\[ \sign\,z = 
   	\frac{z}{|z|}, \quad \text{for }z \in \C\setminus\{0\},
	\]
	and by $\sign 0 = 0.$
\begin{definition} \label{def:signature}
Let $f \in \Sc^\prime(\R;\R)$ and let $\kappa$ be a signature wavelet.
We define the \emph{signature,} $\sigma f,$ of $f$ at $b \in \R$ by
\begin{equation}\label{eq:intro_signature}
    \sigma f(b) = \lim_{a \to 0+} \sign \ps{f,\complexwvlt_{a,b}},
\end{equation}
if the limit exists, and by $\sigma f(b) = 0$ otherwise.
\end{definition}

The convergence in \eqref{eq:intro_signature} is with respect to the standard topology in $\C.$
The signature $\sigma f(b)$ is either equal to zero or it is a complex number of modulus $1.$

The signature may depend on the choice of the signature wavelet,
as the following example illustrates.
	Consider the Weierstraß function (see e.g. \cite{hardy1916weierstrass}) 
	\begin{align*}
		f(x) = \sum_{n=0}^\infty r^n \cos(t^n x), \qquad\text{where $0< r < 1$ and $rt \geq 1,$}
	\end{align*}
and	the zero sequence $a_m = t^{-m},$ $m \in \N.$
	By taking Fourier transforms, we obtain
	\begin{align*}
		\ps{f, \kappa_{a_m,0}} = \pi\sum_{n=0}^\infty { r^n}\ps{\delta_{-t^n} + \delta_{t^n},D_{a^{-1}_m}\widehat{\kappa}}
		= {t^{-\frac m 2}} \pi\sum_{n=0}^\infty { r^n} \widehat\kappa(t^{n- m}).
	\end{align*}
	 If $\supp\hat{\kappa} \subset ]1, t[$, we have $\widehat{\kappa}(t^k) =0$ for all $k \in \Z,$ and consequently 	$\ps{f, \kappa_{a_m,0}} = 0.$
		Thus, if $\sign \ps{f, \kappa_{a,0}}$ converges, then it must necessarily converge to zero.	
	It follows that 		
	\begin{equation} \label{eq:signatureWeierstrass}
		\sigma f (0) = 0 \text{ with respect to $\kappa.$}
		\end{equation}
	Now let $\tilde\kappa$ be a signature wavelet such that 
	 $\hat{\tilde\kappa}(\omega) = 1$ for $\omega \in [1, t].$
	 Then $\ps{f,\tilde\kappa_{a_{m,0}}} > 0$ and  $\sign\ps{f, \tilde\kappa_{a_m,0}} = 1$ 
	 and hence
	 \[ 
		\sigma f (0) = 1, \text{ with respect to $\tilde\kappa.$}
	\]
The example shows that the signature is in general not independent of 
the choice of the wavelet.
However,  we shall see that  the signature does not depend on the signature wavelet in many cases of practical 
importance. Actually, almost all assertions made in this paper are independent of the 
choice of the signature wavelet.
Therefore we omit to the explicit dependance of the employed wavelet in the notation of signature.

\subsection{Basic properties}

We state some basic properties of signature.
Signature 
commutes with translations and with 
 dilations,
i.e.,
\begin{align}\label{eq:translation_invariance}
   \sigma (T_r f) (b) = (\sigma  f) (b -r) \quad  \text{and}\quad  \sigma (D_\nu f ) (b) = (\sigma  f) (\nu b ).
\end{align}
Since the Fourier transform of a signature wavelet $\kappa$ vanishes in a neighborhood of the origin,
we have that 
\begin{align}\label{eq:kappa_polynomial}
	\ps{p, \kappa} = 0, \quad \text{for all polynomials $p.$}
\end{align} 
Therefore, the signature is well defined on the tempered distributions modulo polynomials $\Sc'/ \Pc.$ 
If the signature is non-zero, then the real part or the imaginary part of $\ps{f, \kappa_{a,b}}$
has a constant sign for $a$ small enough.
Another useful property is that
\begin{equation}\label{eq:multiplication}
	\sigma(t f) = \sign (t) \,\sigma f, \quad\text{for all }t \in \C.
\end{equation}
The next property is useful for the computation of signatures. 
\begin{lemma} \label{prop:signature_dominating_term}
 Let $\kappa$ be a signature wavelet and let $b \in \R.$ Further let $f,g \in \mathscr{S}^\prime(\R;\R)$ be two tempered distributions such that $\sigma f (b)\neq 0$ and 
$
    \ps{g,\kappa_{a,b}} \in o\left( \ps{f,\kappa_{a,b}} \right).
$
(Here, the Landau symbol $o$ refers to $a \to 0$.)  Then
\[
    \sigma (f+g) (b) = \sigma f (b).
\]
\end{lemma}
\begin{proof}  
From $\sigma f (b)\neq 0,$  it follows that $\ps{f,\kappa_{a,b}} \neq 0$ for sufficiently small $a.$
Therefore,
\begin{align*}
\frac{\ps{f+g,\kappa_{a,b}}}{\left| \ps{f+g,\kappa_{a,b}} \right|} &= \frac{\ps{f,\kappa_{a,b}} \left(1+ \frac{\ps{g,\kappa_{a,b}}}{\ps{f,\kappa_{a,b}}}\right) }{\left| \ps{f,\kappa_{a,b}} \left(1+ \frac{\ps{g,\kappa_{a,b}}}{\ps{f,\kappa_{a,b}}}\right)  \right|}
\underset{a\to 0}\longrightarrow  \sigma f(b),
\end{align*}
which completes the proof.
\end{proof}
From the above lemma, 
it follows in particular that 	$\sigma(f + g) = \sigma f$ for all bandlimited functions $g.$

\section{The signature at regular points}\label{sec:signatureRegular}

In this section, we compute the signature
 of a signal at some regular points.
We start with some elementary examples.
\begin{example}\label{ex:simple_examples}\hfill
\begin{enumerate}
\item
By \eqref{eq:kappa_polynomial}, 
polynomials have signature equal to zero on the entire real line. 
\item
For the cosine function we get
\begin{align}\label{eq:signature_of_cosine}
\ps{\cos,  \complexwvlt_{a,b}} &= \ps{\widehat{\cos}, (\complexwvlt_{a,b})^\vee}= 
\pi \ps{\delta_1 + \delta_{-1}, (\complexwvlt_{a,b})^\vee} = \pi \sqrt{a} \, e^{-i b} \widehat{\complexwvlt}(a).
\end{align}
Since $\widehat{\complexwvlt}(a) = 0$ for sufficiently small $a,$ 
we obtain
\[
    \sigma(\cos)(b)  = 0,
\]
for all $b \in \R.$
A similar argument applies to the sine function and to arbitrary trigonometric polynomials.
Their signature is identically $0$ as well.
\item In general, since the Fourier transform of a signature wavelet $\kappa$ vanishes in a neighborhood of the origin, 
  we obtain for any band-limited tempered distribution $f$ and for each $b \in \R$,
	\[
		\sign\ps{f, \kappa_{a,b}} = \sign\ps{f^\vee, \widehat{\kappa_{a,b}}} =  0, 
	 \quad \text{for sufficiently small $a.$} 
	\]	
 Hence, the signature of a band-limited tempered distribution is  equal to zero for every $b \in \R.$
\end{enumerate}
\end{example}

In the following, we say that a function $f$ is of \emph{polynomial growth}
if there exists an $N \in \N$ such that $f \in O(|x|^N)$ as $|x| \to \infty.$
Next, we show that a signal of polynomial growth has signature equal to zero 
at a point where all derivatives are equal to zero.   

 \begin{theorem}\label{thm:signature_of_saddle_points2}
	Let $f$ be a real-valued, locally integrable function of polynomial growth. Further assume that $f$ is smooth in a neighborhood of $b \in \R.$
	 If for some $k_0 \in \N_0,$ $f^{(k)}(b) = 0,$ for all $k \geq k_0$,   
 then 
 \[
 	\sigma f (b) = 0.
 \]
 Thus, if $f$ coincides on an open set $U\subset \R$ with a polynomial
then $\sigma f(b) = 0,$  for every $b \in U.$
In particular, we have	$\supp \sigma f  \subseteq \supp f.$
\end{theorem}
We first prove that
functions with infinitely many vanishing moments
 have an unbounded sequence of sign changes.
 For the special case 
of functions whose Fourier transforms vanish in a neighborhood of the origin 
such a statement can be deduced from the results in \cite{eremenko2004oscillation}.
For the general case, we prove the following.
\begin{lemma}
\label{lem:zeros_of_bandpass}
Let $ g \neq 0$ be a real-valued continuous function of rapid decay such that 
\begin{align}\label{eq:vanishing_moments}
	\int_\R g(x) x^k \de x  = 0,  \quad \mbox{ for all } k \in \N_0.
\end{align}
{Then g does not have compact support.}
Moreover,
 there exists an unbounded monotone sequence of sign changes, i.e., there exists $\{x_n\}_{n \in \N}$ such that
$|x_n| \to \infty$, $\sign g(x_{2 n}) = 1$ and $\sign g(x_{2 n+1}) = -1.$ 
\end{lemma}
\begin{proof}
First assume that $g$ has compact support. 
As polynomials are dense in $L^2(\supp g),$ 
$g$ cannot be orthogonal to all polynomials which
contradicts \eqref{eq:vanishing_moments}. This shows the first assertion.

For the second assertion, assume to the contrary that there does  not exist a sequence $\{x_n\}_{n \in \N}$ with the claimed properties. Then we can find an $a > 0$ such that $g$ neither  changes sign on $]-\infty, -a[$ nor on 	$]a, \infty[.$
The continuity of $g$ implies that
    \begin{align} \label{eq:lemma_intermediate}
    		\abs{\int_{-a}^a \p{\frac x a}^m g(x) \de x} \leq 2 a \max_{x \in [-a,a]} |g(x)| =: C, \quad \forall\, m \in \N_0,
    \end{align}
    where $C \geq 0.$
  Since $g$ is continuous and not compactly supported, we find an interval $[c,d]$ in $\R \setminus [-a,a]$ and a constant $K > 0$ such that such that $|g| > K$ on $[c,d].$
 Suppose then, without loss of generality, that $[c,d] \subset ]a, \infty[$ and that 
 $g > K$ on $[c,d].$
    Then there exists an $m_0 \in \N$ such that for $m \geq m_0$ 
    \[
    		\int_{[c,d]} \p{\frac x a}^{m} g(x) \de x > C.
    \]	
    If $g \geq 0$ on $]-\infty, -a[,$ then choose an even $m\geq m_0$,
    and if $g \leq 0$ on $]-\infty, -a[,$ then choose an odd $m \geq m_0$.
    In either case,
    $$
    	\int_{\R \setminus [-a,a]} \p{\frac x a}^{m} g(x) \de x > C,
    $$
    which, together with \eqref{eq:lemma_intermediate}, implies that 
 	\[
		\int_\R \p{\frac x a}^{m} g(x) \de x > 0,
	\]
	contradicting the vanishing moment condition \eqref{eq:vanishing_moments}. This completes the proof.
\end{proof}

\begin{proof}[Proof of Theorem \ref{thm:signature_of_saddle_points2}]
	As signature commutes with translations (cf. \eqref{eq:translation_invariance}),
	we may assume without loss of generality that $b = 0.$
	
	We first assume that $k_0 = 0.$
Since $f$ has polynomial growth,  there exists an $N \in \N$
such that $f \in O(|x|^N)$ as $|x| \to \infty$.
We define a function $H:\R \to \C$ by
	\begin{align}
		H(u) =  u^{N+1}\int_{\R} f(x) \kappa(ux) \de x,
	\end{align}
	and show that $H$ has infinitely many vanishing moments. 
	To this end, we consider the following integrals for $k \in \N_0:$
	\begin{align}\label{eq:interchange_of_integrals}
		\int_\R H(u) u^k \de u = \int_\R  \int_{\R} f(x) \kappa(ux) \de x \, u^{k+N+1} \de u
		= \int_\R  \int_{\R}  \kappa(ux)  u^{k+N+1} \de u \, f(x)  \de x, 
	\end{align}
	Next, we justify the interchange of the order of integration.
	By a change of variables $u = \frac y x,$ for $x \neq 0,$
	we get for the absolute value of the inner integral the following estimate:
	\begin{align*} 
		\int_{\R}  |\kappa(ux)|  |u|^{k+N+1} \de u =  \frac{1}{|x|^{k+N+2}} \int_{\R}  |\kappa(y)|  |y|^{k+N+1} \de y \leq \frac{C_k}{|x|^{k+N+2}},
	\end{align*}
	where $C_k$ is a constant depending only on $k.$
	Therefore, we obtain
	\begin{align} \label{eq:H_inequality}
	 \int_\R  \int_{\R} | \kappa(ux) | \cdot |u|^{k+N+1}\, \de u \, |f(x)| \,\de x \leq C_k \int_\R  \frac{|f(x)|}{|x|^{k+N+2}} \,\de x < \infty, \quad\text{for all }k\in \N_0. 
\end{align}
To justify the last inequality in \eqref{eq:H_inequality}, let  $ \epsilon > 0.$
Then, since $f^{(l)}(0) = 0$ for all $l \in \N_0,$
\[
	\int_{-\epsilon}^\epsilon  \frac{|f(x)|}{|x|^{k+N+2}} \,\de x < \infty.
\]
Furthermore, as $f \in O(x^N),$
\[
\p {\int_{-\infty}^{-\epsilon} + \int_{\epsilon}^\infty } \frac{|f(x)|}{|x|^{k+N+2}} \,\de x < \infty.
\]
Hence, the right hand side of \eqref{eq:interchange_of_integrals} exists in absolute value
	 and thus the interchange of the order of integration in \eqref{eq:interchange_of_integrals} is justified
	  by the Fubini-Tonelli theorem.

	Now, since all moments of the signature wavelet $\kappa$ vanish, we get for $x \neq 0$ that
	\begin{align*}
		\int_{\R}  \kappa(ux)  u^{k+N+1} \de u = 0.
	\end{align*}
	Substitution  into \eqref{eq:interchange_of_integrals} yields
	\begin{align}
	 \int_\R H(u) u^k \de u =  \int_\R  \int_{\R}  \kappa(ux)  u^{k+N+1} \de u \, f(x)  \de x = 0,  \quad \text{for all }k \in \N_0.
	\end{align}
	
	We furthermore observe that $H$ is a continuous function of rapid decay.
	Thus $\real H$ and $\imag H$ are continuous functions of rapid decay all of
	whose moments vanish.
	
If $H = 0$ then $\sign H(u) = \sign\, \langle f, \kappa_{{1/u}, 0}\rangle = 0.$ Hence the fine scale limit of the signs equals $0$ 
 which implies that $\sigma f(0) = 0.$
	If $H \neq 0$ then at least one of the function $ \real H $ and $ \imag H$
	is not identically equal to zero.
	First assume that  both $\real H \neq 0$
	and $\imag H \neq 0.$
	Then  \autoref{lem:zeros_of_bandpass} implies that there is  an unbounded monotone sequence 
$\{u_n\}_{n\in \N}$ of sign changes of $\real H,$ that is, 
\[
 \sign \real H(u_n) = 
 \begin{cases}
 	+1,	& \text{$n$ even;} \\
	-1, & \text{$n$ odd.}
 \end{cases}
\]
Analogously, there exists an unbounded monotone sequence 
$\{v_n\}_{n\in \N}$ of sign changes of $\imag H.$
	Since $\kappa$ is an Hermitian function and $f$ is real-valued, $H$ is also Hermitian, i.e.,
	\[
		H(-u) = \overline{H(u)}.
	\]
By this symmetry property, we may assume that both sequences $\{u_n\}$ and $\{v_n\}$
are increasing and tend to $+\infty.$
Thus, the limit 
$\lim_{u \to \infty} \sign H(u)$
cannot exist. If the imaginary part of $H$ is identically zero 
	then $\sign H(u) = \sign \real H(u),$ thus the limit  $\lim_{u \to \infty} \sign H(u)$ does not exist. Likewise, the limit does not exist if 
	 the real part of $H$ is identically zero.

To conclude the proof, we observe that for $u > 0$ 
	\[
		\sign\, \ps{f, \kappa_{{\frac 1 u}, 0}} = \sign \p{ |u|^{\frac12}\int_\R f(x) \kappa(ux) \de x} = \sign H(u).
	\]
	Thus, with $a = \frac{1}{u},$ the limit
	\[
		\lim_{a\to 0+} \sign \ps{f, \kappa_{a, 0}} = \lim_{u\to \infty} \sign \ps{f, \kappa_{\frac{1}u, 0}} 
	\]
	either does not exist or equals zero.
	In either case, $\sigma f(0) = 0.$
	
	If $k_0 \neq 0,$
we let $p$ be a polynomial of order $k_0-1$ such that $p^{(k)}(0) = f^{(k)}(0)$
for all $k \leq k_0-1.$
Then all derivatives of $f - p$ vanish at $0$
which implies that $\sigma(f - p)(0) = 0.$
Since $\ps{f - p, \kappa_{a,b}} = \ps{f, \kappa_{a,b}},$
we obtain 
$\sigma(f - p)(0) = \sigma f(0) =  0,$
which completes the proof.
\end{proof}

Next we show  that the signature of a homogeneous function vanishes away from the origin. Recall that a homogeneous function of degree $\gamma \in \R$ is a function which satisfies
\[
  f(tx) = t^\gamma f(x), \qquad \text{for all } t > 0.
\]
\begin{proposition}\label{thm:signature_of_homogeneous_function}
 Let $f$ be a homogeneous function of degree $\gamma > -1.$
	Then,
	\begin{align}
		\sigma f (b) = 0, \qquad \text{for all $b \neq 0.$}
	\end{align}
\end{proposition}
\begin{proof} Sine $f$ is a homogeneous function of degree $\gamma > -1$ it is locally integrable on $\R.$
Thus, $f$ defines a distribution on the real line.
By  \cite[Theorems 7.1.16 and 7.1.18]{Hoermander:1990}, $f^\vee$ is homogeneous of degree $-\gamma -1,$ 
and it is a locally integrable function on $\R \setminus \{ 0\},$ 
	Let $\kappa$ be a signature wavelet and let $b \neq 0$ be fixed.
For any $a>0$, we have that
	\begin{align*}
		\ps{f, \kappa_{a,b}} = \ps{f^\vee, \widehat{\kappa_{a,b}}}
		= \sqrt{a}\int_\R f^\vee(\omega)  \hat\kappa(a \omega)e^{-i b \omega} \de\omega 
		= a^{-\frac 1 2}\int_\R f^\vee\p{a^{-1} \xi}  \hat\kappa(\xi) e^{-i b \xi /a} \de \xi.
	\end{align*}
	The integral is well-defined, because $f^\vee$ is locally integrable away from the origin,
	and $\hat\kappa$ is supported away from the origin.
	Hence, 
			\begin{align*}
		\ps{f, \kappa_{a,b}} = a^{\gamma - \frac 1 2}\int_\R f^\vee(\xi)  \hat\kappa(\xi) e^{-i b \xi /a} \de \xi.
	\end{align*}
	Setting $u = \frac{1}{a}$, we obtain
	\begin{align}
	\ps{f,\kappa_{1/u,b}} = u^{-\gamma + \frac 1 2} \int_\R f^\vee(\xi)  \hat\kappa(\xi) e^{-i b \xi u} \de \xi 
	= u^{-\gamma + \frac 1 2} \Fc(f^\vee \, \hat\kappa  )(b u).
	\end{align}
In a neighborhood of the origin,
the function  $\xi \mapsto \widehat{\kappa}(\xi) \, f^\vee(\xi)$ is in $C^\infty$ and vanishes there.  
Arguments  analogous to those in the proof of \autoref{thm:signature_of_saddle_points2} imply that
 the function	 $u \mapsto \Fc(f^\vee \, \hat\kappa  )(b u)$ is of rapid decay and has infinitely many vanishing moments.  Finally, by \autoref{lem:zeros_of_bandpass}, we obtain that $\Fc(f^\vee \, \hat\kappa  )(b u)$
has an unbounded sequence of zeros.
 Hence, 
\[
	\sign\ps{f, \kappa_{1/u,b}} = \sign \left( u^{-\gamma+\frac12}  \Fc(f^\vee \, \hat\kappa  )(b u)
	\right) 
\]
cannot converge for $u \to \infty$ to a non-zero value. It follows that $\sigma f (b) = 0,$ 
which completes the proof.
 \end{proof}

\section{The signature at singular points}\label{sec:signatureSingular}

In this section we compute the signature at
singular points of a signal, such as cusp and jumps.
To this end, we will need the following lemma.

\begin{lemma}\label{lem:waveletToZero}
   Let $g$ be locally integrable of polynomial growth and continuous at $0$ with $g(0) = 0$.
 Let  $h$ be a homogeneous function of degree 
  $\gamma \geq 0.$ Then we have for each $\psi \in \Sc(\R)$ that
	\[
		\frac{1}{a^{\gamma+1}}\int_{\R} |h(t)g(t) \psi\p{\tfrac{t}{a}}| \to 0 \quad \text{ as } \quad a \to 0^+. 
	\]
\end{lemma}
\begin{proof}
Let $\epsilon >0$ be arbitrary.
Since $g(0) = 0$ and since $g$ is continuous around the origin,
we can find $\delta > 0$  such that $|g(x)| \leq \epsilon$ for all $|x| \leq \delta.$
Thus, around the origin we get the estimate
\begin{align*}
	\frac{1}{a^{\gamma+1}} \abs{ \int_{-\delta}^\delta g(x) h(x)\psi\p{\frac{x}{a}} \de x} 
	&\leq \frac{\epsilon}{a^{\gamma+1}} \int_{-\delta}^\delta \abs{h(x)\psi\p{\frac{x}{a}}} \de x \\
	&= \epsilon \int_{-\frac{\delta}a}^{\frac{\delta}a} \abs{h(y)\psi\p{y}} \de y
	\leq \epsilon \| h\psi \|_{L^1}.
\end{align*}
Here we used the homogeneity of $h$ and the integrability of $h\psi.$
Since $g$ is of polynomial growth and 
$\psi$ is of rapid decay, there exist  $C, M > 0$ and
 $N > \gamma$  such that 
 \[
 |h(x)g(x)| \leq C |x|^N \quad\text{and}\quad |\psi(x)| \leq C |x|^{-N-2},
 \]
 for all $|x| \geq M.$
Hence, there is $a_0 > 0$ such that
\[
	\frac{1}{a^{\gamma+1}} \abs{\int_{|x| > M} h(x) g(x) \psi\p{\frac{x}{a}} \de x}
	\leq \frac{C^2}{a^{\gamma+1}} \abs{\int_{|x| > M} |x|^N \frac{|a|^{N+2}}{|x|^{N+2}} \de x}
	\leq  \epsilon,
\]
for all $0 < a < a_0.$ By the rapid decay property of $\psi,$
there exists  $0 < a_1 \leq a_0,$
such that 
\[ 
	\frac{1}{a^{\gamma+1}} \sup_{[\delta, M]} \abs{\psi\p{\frac{x}{a}}} = 	\frac{1}{a^{\gamma+1}} \sup_{[a \delta, a M]} \abs{\psi\p{x}}  \leq \epsilon,
\] 
for all $0< a \leq a_1.$ As $g$ is locally integrable, it follows that
\[
	\frac{1}{a^{\gamma+1}} \abs{\int_{\delta < |x| < M}  h(x)g(x) \psi\p{\frac{x}{a}} \de x} \leq 	\frac{1}{a^{\gamma+1}}\int_{-M}^M |h(x)g(x)| \de x \cdot \sup_{[\delta, M]} \abs{\psi\p{\frac{x}{a}}} \leq C' \epsilon,
\] 
with a constant $C'$ that is independent of $a.$ Combining all three  estimates,
we have found that for every $\epsilon > 0$ there exists an $a_1 > 0$ such that 
\[
	\frac{1}{a^{\gamma+1}} \abs{\int_\R h(x)g(x) \psi\p{\frac{x}{a}} \de x} <  (\| h\psi\|_{L^1} +  1 + C') \, \epsilon,\quad \text{for all }a < a_1.
\]
This proves the assertion.
\end{proof}

\subsection{Jump singularities}

We compute the signature at jump discontinuities. 
We say that a function $f$ has a \emph{jump (or step) discontinuity} 
at $b$ if the left-hand and the right-hand limits 
$f(b^+)$ and $f(b^-)$ exist but are not equal.
 We also let $U$  be the unit step (or Heaviside) function given by
\[
   U(x) = 
   \begin{cases} 
   1, &\mbox{if } x \geq 0, \\ 
   0, &\text{else.}
   \end{cases}
\] 
We obtain the following result for jump discontinuities.

 \begin{theorem}\label{prop:signature_of_step}
  Let $f$ be a real-valued, locally integrable function of polynomial growth
 and let $b \in \R.$
  If there exists a neighborhood $V$ of $b$ such that $f$ is continuous on $V\setminus 
\{b\}$ and has a jump discontinuity at $b,$
 then 
\begin{equation} \label{eq:sign_jump}
\sigma f(b) =  
\begin{cases}
 +i, &\text{if}\quad  f(b^-) < f(b^+), \\
 -i, &\text{if}\quad f(b^-) > f(b^+). 
 \end{cases}
\end{equation}
 \end{theorem} 
  \begin{proof}
  	By  the commutation property with translations 
 \eqref{eq:translation_invariance}, we may assume that $b = 0$. 
 Further, as subtracting a constant does not alter the signature,
 we may assume that $\frac{f(0^+) + f(0^-)}{2} = 0.$
By \eqref{eq:multiplication}
we have $\sigma(-f) = -\sigma f,$
hence it is sufficient to prove the results for the case $\jump = f(0^+) - f(0^-) > 0.$
 Throughout this proof we write $f^e$ for the even part 
 and $f^o$ for the odd part of $f,$
 that is, 
 \begin{align*}
 	f^e(x) = \frac{f(x) + f(-x)}{2} \quad\text{and}\quad f^o(x) = \frac{f(x) - f(-x)}{2}.
 \end{align*}

    Let $\kappa$ be a signature wavelet. 
Since $\hat\kappa$ is real-valued,
$\real\kappa_{a,0}$ is even and $\imag\kappa_{a,0}$ is odd.
Thus we have
 \begin{align*}
     \real\ps{f,\complexwvlt_{a,0}} &=  \ps{f^e + f^o, \real\kappa_{a,0}} 
     = \ps{f^e,\real\kappa_{a,0}}
 \end{align*}
 and  
 \begin{align*}
     \imag\ps{f,\complexwvlt_{a,0}} &=  \ps{f^e + f^o, \imag\kappa_{a,0}}
     = \ps{f^o,\imag\kappa_{a,0}}.
 \end{align*}
 We observe that
 \[
	 C_\kappa = \int_0^\infty \imag \kappa(x) \de x = \imag 
	 \int_\R U(x) \kappa(x) \de x =
	 \imag \int_\R -\frac{1}{i \omega} \hat\kappa(\omega) \de \omega = 
	   \int_\R \frac{1}{\omega} \hat\kappa(\omega) \de \omega > 0.
	\]
 By \autoref{lem:waveletToZero} we get 
 \begin{align}\label{eq:jump_proof_1}
	\frac1{\sqrt{a}}\ps{f^e, \real \kappa_{a,0}} =  \frac{1}a\int_\R f^e(x) \real\kappa\p{\frac{x}{a}} \de x \to 0, 
	\quad\text{as $a \to 0$}.
	\end{align}
	We next show that 
	 \begin{align}\label{eq:jump_proof_2}
	\frac1{\sqrt{a}}\ps{f^o, \imag \kappa_{a,0}} =  \frac{1}a\int_\R f^o(x) \imag\kappa\p{\frac{x}{a}} \de x \to C_\kappa \jump , 	\quad\text{as $a \to 0,$}
	\end{align}
	which is  equivalent to
	 \begin{align*}
 \frac{1}a\abs{\int_\R f^o(x) \imag\kappa\p{\frac{x}{a}} \de x - C_\kappa \jump} 
 \to 0, 	\quad\text{as $a \to 0$}.
	\end{align*}
	By the antisymmetry of $\imag \kappa,$ this can be rewritten as
	\[
	\frac{1}a\abs{\int_\R f^o(x) \imag\kappa\p{\frac{x}{a}} \de x  - \jump \int_0^\infty \imag \kappa(x)\de x} = \frac{1}a\abs{\int_\R \p{f^o(x) - \sign(x)\frac{\jump}{2}}\imag\kappa\p{\frac{x}{a}} \de x}.
	\]
	Since $f^o(x) - \sign(x)\frac{\jump}{2}$ is continuous around the origin,  \eqref{eq:jump_proof_2}
	follows from \autoref{lem:waveletToZero}.

From \eqref{eq:jump_proof_1} and \eqref{eq:jump_proof_2},  it follows that
\[
     \real \ps{f,\kappa_{a,0}} \in o \left( \imag \ps{f,\kappa_{a,0}} \right), \quad\text{for  } a \to 0.
\]
Thus we get by \autoref{prop:signature_dominating_term} that
\[
    \lim_{a \to 0} \sign \ps{f,\kappa_{a,0}} =  i,
\]
	which completes the proof.
  \end{proof}

Now are able to compute the signature of the  unit step $U$.
Using \autoref{thm:signature_of_saddle_points2} for $b \neq 0$ 
and  \autoref{prop:signature_of_step} for $b = 0$ we get
\begin{equation}\label{eq:signatureHeaviside}
	\sigma U (b) = 
	\begin{cases} 
   i, &\mbox{if } b = 0, \\ 
   0, &\mbox{else.} 
   \end{cases}
\end{equation}
The same arguments can be applied to the sign function $x \mapsto \sign x.$

\subsection{Cusp singularities}

We say that $f$ has a symmetric cusp singularity of order $\gamma \in (0, 1]$ at $x_0 \in \R$ 
if 
\begin{equation}\label{eq:defCusp}
  f(x) = \pm |x-x_0|^\gamma \, (1 + g(x)),
 \end{equation}
where $g$ is a locally integrable function with $\lim_{x \to x_0} g(x) = 0.$ 
We say that the symmetric cusp is downwards if the positive sign holds in \eqref{eq:defCusp}
and upwards if the negative sign holds in \eqref{eq:defCusp}.
Similarly, we say that $f$ has a right hand cusp of order $\gamma \in (0, 1]$  at $x_0$ if 
\begin{equation}\label{eq:defCuspOneSided}
 f(x) =  (x-x_0)_+^\gamma \, (1 + g(x))
 \end{equation}
where  
 \[ 
  x_+^\gamma = \begin{cases}
x^\gamma & \text{for } x \geq 0,\\ 
0 &\text{else.}
 \end{cases}
 \]
Likewise we call the reversed version a left hand cusp.

\begin{theorem}\label{thm:signatureCusp}
Let $f$ be a real-valued, locally integrable function of polynomial growth.
 If  $f$ has a symmetric cusp singularity of order $\gamma \in (0, 1]$ 
 at $b \in \R$ then
 	 \begin{equation}
 \sigma f(b) =
 \begin{cases}
   +1, &\text{for an upward cusp,}\\
 -1, &\text{for a downward cusp.} 
 \end{cases}
	 \end{equation}
	 For a one-sided cusp we have
	  	 \begin{equation}
 \sigma f(b) =
 \begin{cases}
 e^{+i (\gamma+1) \pi/2}, &\text{for a right hand cusp,}\\ 
  e^{-i(\gamma+1) \pi/2}, &\text{for a left hand cusp.}
 \end{cases}
	 \end{equation}
 \end{theorem}
\begin{proof}
Since signature commutes with translations, we may assume that $b = 0.$

By \eqref{eq:defCusp} and \eqref{eq:defCuspOneSided},
$f$ is of the form 
\[
	f(x) = h(x) (1+ g(x))
\]
with $h = \pm |x|^\gamma$ and $h(x) =  \pm (x)_+^\gamma,$ respectively.
Let us first assume that $f$ is a pure cusp which means that $g = 0.$
If $\gamma=1$  then, according to \cite[eq. 2.3(19)]{gelfand1964generalized}, we obtain
\begin{align} \label{eq:cusps1}
	\ps{ |\argdot|^\gamma, \kappa_{a, 0}}  = {\ps{ \widehat{|\argdot|^\gamma}, (\kappa_{a, 0})^\vee} } 
	= 2(-1)^{\frac{\gamma -1}{2} +1} \gamma! \int_\R \frac{(\kappa_{a, 0})^\vee(\xi)}{|\xi|^{\gamma+1}}   \de{\xi}.
	\end{align}
For $0 < \gamma <1,$ 
it follows from  \cite[eq. 2.3(12)]{gelfand1964generalized} that 
	\begin{align} \label{eq:cusps2}
	\ps{ |\argdot|^\gamma, \kappa_{a, 0}}  = {\ps{ \widehat{|\argdot|^\gamma}, (\kappa_{a, 0})^\vee} } 
	= - \sin\left(\frac{\pi\gamma}{2}\right) \Gamma(\gamma+1) \int_\R \frac{(\kappa_{a, 0})^\vee(\xi) }{|\xi|^{\gamma+1}}  \de{\xi}.
	\end{align}
In both cases, it follows from a change of variables that there is $C_\gamma >0$ such that
\begin{equation}\label{eq:proofCusp1}
	 \ps{ |\argdot|^\gamma, \kappa_{a, 0}} = - C_\gamma a^{\gamma+1} < 0.
\end{equation}
It follows that
\[
	\sigma (|\argdot|^\gamma)(0) = -1.
\]
and by \eqref{eq:multiplication} that
\[
	\sigma (-|\argdot|^\gamma)(0) = 1.
\]
Similarily, we have for the right hand cusp that (see, e.g. \cite[eq. (1.5)]{unser2000fractional})
\begin{equation}\label{eq:rhCusp}
\begin{split}
\ps{ (\argdot)_+^\gamma, \kappa_{a, 0}}  &= {\ps{ \widehat{(\argdot)_+^\gamma}, (\kappa_{a, 0})^\vee} } \\
	&= \Gamma(\gamma+1) \int_\R \frac{(\kappa_{a, 0})^\vee (\xi)}{(i\xi)^{\gamma+1}} \de{\xi} \\  
	&= e^{i (\gamma+1) \pi/2}  \Gamma(\gamma+1) \int_{-\infty}^0 \frac{(\kappa_{a, 0})^\vee (\xi)}{|\xi|^{\gamma+1}} 
	\de{\xi} \\
	&= e^{i(\gamma+1) \pi/2} C_\gamma' a^{\gamma+1}
	\end{split}
\end{equation}
with $C'_\gamma >0.$
Hence, we obtain
\[
	\sigma( (\argdot)_+^\gamma ) (0) = e^{i(\gamma+1) \pi/2}.
\]
For the reversed version, i.e. the left hand cusp, we get by Hermitian symmetry
\[
	\sigma( (-\argdot)_+^\gamma ) (0) = e^{-i(\gamma+1) \pi/2}.
\]

Now let us turn to the case where $g$  is not identically $0.$
We notice that in all cases $h$ is a homogeneous function of degree $\gamma.$
By \autoref{lem:waveletToZero} we get that
\[
	\frac{1}{a^{\gamma+1}} \ps{ h \,g, \kappa_{a, 0}} \to 0.
\]
From \eqref{eq:proofCusp1} and \eqref{eq:rhCusp},
it follows that
 $\ps{h \, g, \kappa_{a, 0}} \in o( \ps{g ,\kappa_{a, 0}} ).$
 Since $\sigma (h)(0) \neq 0$, we can apply   \autoref{prop:signature_dominating_term} to get
 \[
	\sigma (f)(0) = \sigma (h +  g h) = \sigma (h)(0),
\]
 which completes the proof.
\end{proof}

\begin{figure}
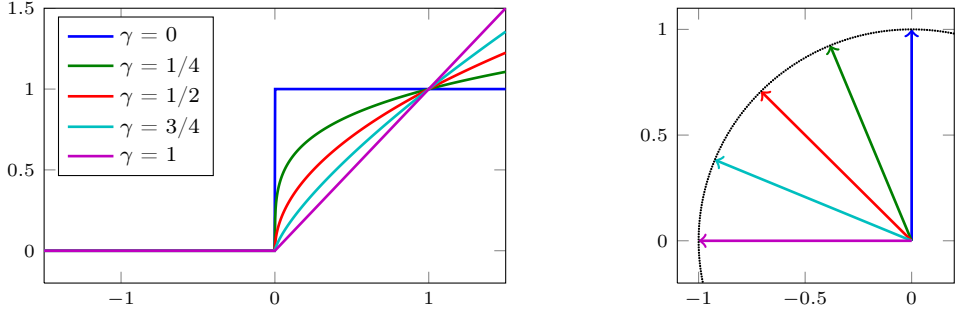

\def\thisfigscaling{0.45\textwidth}
\def\thisfigheight{0.9\textwidth}
\def\thisfigwidth{0.9\textwidth}
	\def\figurewidth{1\textwidth}
	\def\figureheight{1\textwidth}
		\def\plottitle{}
\centering
\begin{subfigure}[t]{\thisfigscaling}\centering
	\def\figurewidth{1\textwidth}
	\def\figureheight{0.6\textwidth}
	\input{matlab/cuspSingFigRight/cuspSingPlot.tikz}
	\end{subfigure}
	\hfill
		\def\figurewidth{0.6\textwidth}
	\def\figureheight{0.6\textwidth}
	\begin{subfigure}[t]{\thisfigscaling}\centering
		\input{matlab/cuspSingFigRight/cuspSingSign.tikz}
			\end{subfigure}
		\caption{
		\emph{Left:} Right hand cusps of power-type of the form $f(x)= x_+^\gamma$ for different values of $\gamma.$
		\emph{Right:} Corresponding signatures at the origin $\sigma f(0)$ as orientations in the complex plane. 
		The signature evolves from purely imaginary (unit step, $\gamma = 0$) to purely real (ramp function, $\gamma = 1$).
		}
		\label{fig:cuspsOneSided}
\end{figure}

A particular example of the above theorem is a pure downwards cusp of power-type
$f(x) = |\argdot|^\gamma.$
Using \autoref{thm:signature_of_saddle_points2} for $b \neq 0$ 
and  \autoref{thm:signatureCusp} for $b = 0$ we get
\begin{equation}\label{eq:signatureCusp}
	\sigma f (b) = 
	\begin{cases} 
   -1, &\mbox{if } b = 0, \\ 
   0, &\mbox{else.} 
   \end{cases}
\end{equation}
Figure~\ref{fig:cuspsOneSided} shows the signature of right hand cusps
for different values of $\gamma.$

\section{The signature under fractional Hilbert transforms}\label{sec:signatureHilbert}

Next, we  study the behavior of the signature under fractional Hilbert transforms.
The fractional Hilbert transform was first introduced in \cite{lohmann1996fractional}. We follow the definition given in \cite{luchko2008fractional}.
For $\alpha \in \R$, the fractional Hilbert transform $\hilbert^\alpha$ is defined on 
$\SmodP$ by
\begin{equation} \label{eq:fractional_Hilbert}
 \widehat{\hilbert^\alpha f} = e^{- i\alpha \frac{\pi}{2}  \, \sign (\argdot)} \, \widehat{f}. 
\end{equation}
For $\alpha =1$ we obtain the classical Hilbert transform $\Hc = \Hc^1$.
We next show that the fractional Hilbert transform $\Hc^\alpha$ acts on the signature
as a multiplication by $e^{i\alpha \frac \pi 2},$ i.e., as a rotation in the complex plane.
\begin{theorem}\label{thm:frac_hilbert_and_signature}
Let $f \in \SmodP$ and $b \in \R.$ 
Then 
\begin{align}
   \sigma \left( \hilbert^\alpha f \right) (b) = e^{i \alpha \frac{\pi}{2}} \, \sigma f (b).
\end{align}
\end{theorem}

\begin{proof} 
By  the translation invariance \eqref{eq:translation_invariance}, it is sufficient to  prove the result for $b = 0$. Let $a>0$ and let $\kappa$ be a signature wavelet. We have
\begin{align*}
\ps{\hilbert^\alpha f,\kappa_{a,0}} &= \ps{e^{-i \alpha \frac{\pi}{2} \sign (\omega)} \hat f, (\kappa_{a,0})^\vee }
=  \int_\R e^{-i \alpha \frac{\pi}{2} \sign (\omega)} 
 \hat{f} (\omega) (\kappa_{a,0})^\vee(\omega) \de \omega.
\end{align*}
Since $\supp \kappa_{a,0}^\vee \subset (-\infty,0],$ this reduces to
\begin{align*}
\ps{\hilbert^\alpha f,\kappa_{a,0}} 
&=  \int_{-\infty}^0 e^{-i \alpha \frac{\pi}{2} \sign (\omega)} \hat{f}(\omega)  (\kappa_{a,0})^\vee(\omega) \de \omega \\
&= e^{i \alpha \frac{\pi}{2}}  \int_{-\infty}^0  \hat{f}(\omega)  (\kappa_{a,0})^\vee(\omega) \de \omega 
= e^{i \alpha \frac{\pi}{2}}  \ps{\widehat{f},(\kappa_{a,0})^\vee} = e^{i \alpha \frac{\pi}{2}} \ps{f,\kappa_{a,0}}.
\end{align*}
Thus, we obtain
\begin{align*}
   \sigma\left( \hilbert^\alpha f \right)(0) = 
   \lim_{a \to 0}  \sign  \ps{\hilbert^\alpha f,\kappa_{a,0}}
   =  e^{i \alpha \frac{\pi}{2}} \lim_{a \to 0}  \sign  \ps{ f,\kappa_{a,0}}
   = e^{i \alpha \frac{\pi}{2}} \sigma f (0),
\end{align*}
which proves the claim.
\end{proof}
An interesting case of \autoref{thm:frac_hilbert_and_signature} is the following example.
	We consider the sign function $f(x) = \sign(x).$
	Its (ordinary) Hilbert transform is given by $\Hc f(x) = 2\log |x|$ in the distributional sense,
	 cf. \cite[p. 80]{meyer1992wavelets}.
	 \autoref{thm:frac_hilbert_and_signature} and \eqref{eq:signatureHeaviside} imply that
	\[
		\sigma (\log|\argdot|) (b) = \sigma (\Hc f) (b) = i \, \sigma (f) (b) =
		\begin{cases}
		 	i^2 = -1, &\text{if }b = 0, \\
			0, &\text{else}. 
		 \end{cases}
	\]
	Thus the logarithm has the same signature as the cusp of power type \eqref{eq:signatureCusp}.

An illustration of the action of the Hilbert transform 
on the signature is given in Figure~\ref{fig:fracHilbert}.

\begin{figure}
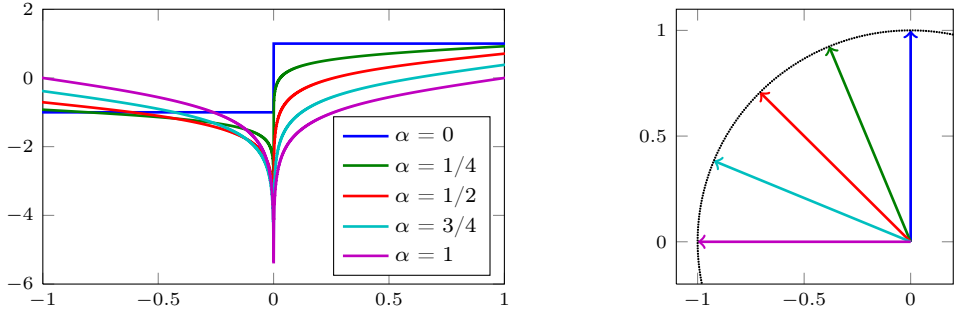

\def\thisfigscaling{0.45\textwidth}
\def\thisfigheight{0.9\textwidth}
\def\thisfigwidth{0.9\textwidth}
	\def\figurewidth{1\textwidth}
	\def\figureheight{1\textwidth}
		\def\plottitle{}
\centering
\begin{subfigure}[t]{\thisfigscaling}\centering
	\def\figurewidth{1\textwidth}
	\def\figureheight{0.6\textwidth}
	\input{matlab/fracHilbert/hilbertPlot.tikz}
	\end{subfigure}
	\hfill
		\def\figurewidth{0.6\textwidth}
	\def\figureheight{0.6\textwidth}
	\begin{subfigure}[t]{\thisfigscaling}\centering
		\input{matlab/fracHilbert/hilbertSign.tikz}
			\end{subfigure}
		\caption{
		\emph{Left:} Fractional Hilbert transforms of the sign function  $\Hc^{\alpha}\sign(\argdot)$ for different values of $\alpha.$
		\emph{Right:} Corresponding signature at the origin $\sigma f(0)$ as orientations in the complex plane. 
		The signature evolves from purely imaginary (sign function, $\alpha = 0$) to purely real (logarithm, $\alpha = 1$).
		}
		\label{fig:fracHilbert}
\end{figure}

\section{The signature of random signals}\label{sec:signatureRandom}

In this section, we show that certain random signals  
do not contain any salient points in the sense of signature.

\subsection{Gaussian white noise and fractional Brownian motion }

We briefly recall the definition of Gaussian white noise 
and fractional Brownian motion. For a detailed description, we refer to the 
textbooks \cite{hida1980brownian, unser2014introduction}. 
A Gaussian white noise with variance $\sigma^2$ is determined by the characteristic  functional $\widetilde P$, defined for $\phi \in \Sc(\R)$ by
\[
	\widetilde P (\phi) =  e^{-\frac{\sigma^2}{2}\norm{\phi}_2^2}.
\] 
By the Minlos-Bochner theorem, there is a unique probability measure $P$ on $\Sc'(\R)$
such that for all $\phi \in \Sc(\R)$
\[
	\int_{\Sc'(\R)} e^{i \ps{\phi, f}}\, dP(f) = \widetilde{P}(\phi).
\]
A random variable $w$ that follows the probability law $P$ is called Gaussian white noise with variance $\sigma^2$. 
In order to define  the fractional Brownian motion, let $D^\gamma$ be the fractional differential operator of order $\gamma \in \R$ 
defined for $f \in \Sc'(\R)/\Pc$ 
by
 \[
\widehat{D^\gamma f}(\omega) = (\ci \omega)^\gamma \, \hat f (\omega).
\]
By definition, a fractional Brownian motion $B_\gamma$ of order $\gamma$ is a solution of the stochastic differential equation
\begin{equation}\label{eq:brownianMotion}
	D^{\gamma} B_\gamma  = w,
\end{equation}
where $w$ is a Gaussian white noise.  With this notation $B_0=w$ and $B_{-1}$ is the classical Brownian motion.

We next show that Gaussian white noise and fractional Brownian motion have a vanishing signature.
In order to prove this result we need the following two elementary properties of a generalized Gaussian noise \cite[Chapter 3.4]{hida1980brownian}: For each Schwartz function $\phi\neq 0$ we have that $\ps{w, \phi}$ 
is a Gaussian random variable with zero mean and with variance $\sigma^2 \|\phi\|^2_2.$
Further, if $\phi$ and  $\psi$ are orthogonal with respect to the $L^2$ inner product
 then the random variables $\ps{w, \phi}$ and $\ps{w, \psi}$ are independent.

\begin{theorem}
Let $w$ be a Gaussian white noise with variance $\sigma^2$. For all $b \in \R$ we have
\[
	 \sigma (w) (b) = 0 \text{ with probability $1.$}
\]
More generally,  let $B_\gamma$ be a fractional Brownian motion of order $\gamma \geq 0.$
For all $b \in \R$ we have
\[
	 \sigma (B_\gamma) (b) = 0 \text{ with probability $1.$}
	\]
\end{theorem}
\begin{proof}
Let $\kappa$ be a signature wavelet. By the commutation with translation property it is sufficient to prove the assertion for $b =0.$
We first treat the case $\gamma = 0$
where $B_\gamma = w.$
Since $w$ is a Gaussian white noise, we have that
 both the real part and the imaginary part of $\ps{w, \kappa_{a,0}}$ are Gaussian distributed
 with zero mean and the variance $\sigma^2 \| \real \kappa_{a,0} \|_2^2 = \sigma^2  \| \Hc \real \kappa_{a,0} \|_2^2 =  
 \sigma^2 \| \imag \kappa_{a,0} \|_2^2.$
The real and the imaginary part are independent due to the orthogonality of the real part and the imaginary part of $\kappa,$ that is, 
\[
	\ps{\real\kappa_{a,0},  \imag\kappa_{a,0}} = 	\ps{\real\kappa_{a,0},  \Hc\real\kappa_{a,0}} = 0.
\]
	Hence, $\sign\ps{w, \kappa_{a,0}}$ 
	follows a uniform distribution on the complex unit circle.
Let $(a_j)_{j\in \N}$ be a decreasing zero sequence
such that $\hat\kappa_{a_j,0}$ and $\hat\kappa_{a_k,0}$ have disjoint supports for all $j \neq k.$
By Parseval's relation it follows that $\kappa_{a_j,0}$ and $\kappa_{a_k,0}$ are orthogonal with respect to the $L^2$ inner product.
It follows that $\sign\ps{w, \kappa_{a_j,0}}$ and $\sign\ps{w, \kappa_{a_k,0}}$ are independent.
Now let $I$ be an arc on the complex unit circle of length $0< \epsilon <\pi.$
Since  $\ps{w, \kappa_{a_j,0}}$ is uniformly distributed on the unit circle,
 we get for all $j \in \N$
\[ 
P(\sign\ps{w, \kappa_{a_j,0}} \in I) \leq \frac{1}{2}.
\]
 From the independence for $k \neq j$ it follows 
 for all index sets $J \subset \N$ with $|J| = N$ 
 that
\[
	P( \{ \sign\ps{w, \kappa_{a_j,0}} \in I \text{ for all } j \in J \}) \leq 2^{-N} \to 0.
\]
Hence, the probability that $\sign\ps{w, \kappa_{a_j,0}}$ converges for $j \to \infty$ is equal to zero.
This shows the assertion for $\gamma = 0.$

For $\gamma \neq 0,$ we get by 
 \eqref{eq:brownianMotion} that
\begin{align*}
	\sign \ps{\Hc^{\gamma} B_\gamma, \kappa_{a,0}} &=  
	\sign \ps{\Hc^{\gamma} D^{-\gamma}  w, \kappa_{a,0}} \\
	&= 
		\sign \ps{ (-\Delta)^{-\gamma/2} w,  \kappa_{a,0}} \\
		&= 
		\sign \ps{  w,  (-\Delta)^{-\gamma/2} (\kappa_{a,0})}\\
		&= 
		\sign \ps{  w,  ((-\Delta)^{-\gamma/2} \kappa)_{a,0}}.
\end{align*}
Since $(-\Delta)^{-\gamma/2} \kappa$ is a signature wavelet, 
we can use the same arguments as before to conclude that $\sigma(\Hc^{\gamma} B_\gamma)(0)~=~0.$
Using \autoref{thm:frac_hilbert_and_signature} 
we get
\[
\sigma B_\gamma(0) = e^{-i \gamma \frac{\pi}{2}}  \sigma(\Hc^{\gamma} B_\gamma)(0) = 0,
\]
which completes the proof.
\end{proof}

\subsection{Randomization of wavelet coefficients}

Let $\{\psi_{jk}\}_{j,k \in \Z}$ be an 
orthonormal wavelet basis of $L^2(\R)$, where $\psi
:=\psi_{00}$ has  compactly supported Fourier transform. Here $j$
stands for the dyadic level and $k$ for the translation parameter. For
$f \in L^2(\R)$, we have
the following decomposition
\begin{equation*} 
     f = \sum_{j \in \Z} \sum_{k \in \Z} \ps{f,\psi_{jk}} \psi_{jk} =
\sum_{j \in \Z} \sum_{k \in \Z} f_{jk} \psi_{jk},
 \end{equation*}
where convergence is in $L^2$-norm.
  We define an operator $A_\varepsilon: L^2(\R) \to L^2(\R)$ of
random perturbations of the wavelet signs
      by
 \begin{equation} \label{eq:rand_series}
     A_\varepsilon f := \sum_{j \in \Z} \sum_{k \in \Z}
\varepsilon_{jk} f_{jk} \psi_{jk},
\end{equation}
where  $\{\varepsilon_{jk}\}$ is a Rademacher sequence, i.e. a sequence
of  independently and identically distributed random variables taking
the values $\pm1$ with probability $\frac12$. Regarding the use of
Rademacher sequences in the context of random Fourier series and random matrices, see
\cite{Kahane:1985} and \cite{rauhut2012restricted}. More recently, Jaffard
    \cite{Jaffard:2005} considered the use of Rademacher sequences for
random perturbations
    of wavelet series.  Our next theorem shows that  the signature of
$A_\varepsilon f$ vanishes for almost every realization of the
Rademacher sequence.

\begin{theorem}\label{thm:jaffard}
Let $f \in L^2(\R),$ $b\in \R,$ and $A_\varepsilon$ an operator of random sign perturbations as defined  in \eqref{eq:rand_series}.
Then
\[
   \sigma(A_\varepsilon f)(b)  = 0,\mbox{ for almost every } \varepsilon.
\]
\end{theorem}

\begin{proof}
We start by observing that
\begin{equation} \label{eq:eq17}
   \sign   \ps{A_\varepsilon f,\kappa_{a,b}}  = \sign \left(
\sum_{l,m} \varepsilon_{l,m} \ps{f,\psi_{l,m}}
\ps{\psi_{lm},\kappa_{a,b}} \right), \mbox{ for any } a > 0.
\end{equation}
We recursively construct a decreasing sequence $a_j$ such that $\lim_{j\to \infty} a_j = 0$  as follows.
Let us denote by $S_j$ the set of indices such that
$\ps{\psi_{lm},\kappa_{a_j,b}}$ does not vanish, i.e.
$$
    S_j := \{ (l,m) \in \Z^2 : \ps{\psi_{lm},\kappa_{a_j,b}} \neq 0 \}.
$$
As the Fourier transform of $\psi_{lm}$ is compactly supported and
because of the structure of the compact supports of
 the Fourier transforms of the functions $\kappa_{a_j,b}$, we can 
choose the sequence $\{a_j\}_j$
such that $S_j \cap S_{i} = \emptyset$ for all $j\neq i.$
We consider the corresponding sequence of random variables
 \[
(A_\varepsilon f)_{j} := \sum_{l,m} \varepsilon_{l,m} \ps{f,\psi_{lm}}
\ps{\psi_{lm},\kappa_{a_j,b}}.
\]
The pairwise disjointness of the $S_j$ implies that
 $(A_\varepsilon f)_{j,b}$ consists  of independent random
variables. 
Now since $\{\varepsilon_{lm}\}$
is a Rademacher sequence,  we have that
\[
    \sign (A_{-\varepsilon} f)_{j} = -\sign (A_{\varepsilon} f)_{j}.
\]
Let $I_1$ be the
 arc of length $\pi$ centered at $1,$ i.e.,
$I_1 = \{ z \in \C : |z| = 1, \real z \geq 0\}.$
For all $j > 0$, we have  that
\begin{equation} \label{eq:proba_one_j}
   P(\sign(A_\varepsilon f)_{j}\in I_1) = P(\sign(A_\varepsilon f)_{j}\in
-I_1) = \frac{1}{2}
\end{equation}
where $P$ denotes the probability measure associated with the Rademacher
sequence $\{\varepsilon_{lm}\}.$
It follows from \eqref{eq:proba_one_j} and from the independence of 
 for all index sets $J \subset \N$ with $|J| = N$ 
 that
\[
	P( \{ \sign(A_\varepsilon f)_{j} \in I_1 \text{ for all } j \in J\}) \leq 2^{-N} \to 0.
\]
An analogous assertion is true for the arc  $I_2$ centered around $i$ given by
$I_2 = \{ z \in \C : |z| = 1, \ \imag{z} \geq 0\}.$
Hence, we have
\[
\sigma(A_\varepsilon f)(b)  = 0, \quad\text{with probability 1.}\qedhere
\]
\end{proof} 

Let us illustrate the above result by considering the characteristic function on $[0,1]$: 
$f = \indfunc_{[0,1]}.$
The signal $f$ has jump discontinuities at $0$ and $1,$ but the shape of the randomly perturbed signal $A_\varepsilon f$ 
 does not have much in common with the original jump shape of $f.$
 Jaffard has shown \cite{Jaffard:2005} that 
 $A_\epsilon f$ is unbounded with probability $1.$
However, 
since the amplitude of the wavelet coefficients is not altered,
$\ps{f,\psi_{jk}} = \ps{A_\varepsilon f,\psi_{jk}},$ 
it is difficult  to distinguish between $f$ and $A_\varepsilon f$ using a local signal analysis tool that is based only on the moduli of the wavelet coefficients.  
Signature, on the other hand, takes into account the random perturbations.
The signature of the original function $f$ is non-zero at $0$ and $1,$
but identically zero for the perturbed function $A_\varepsilon f.$ 

\section{Relations to the singular support and a geometric interpretation}
\label{sec:signatureGeometric}

In the examples that we have considered so far, the points of non-zero signature are a subset of the classical singular support.
This gives rise to the  question whether there is a relation between the classical singular support 
 and the support of the signature $\supp \sigma f.$
 Recall that a point $x_0$ is not in the singular support of a distribution $f$
 if there is a compactly supported $C^\infty$ function $\phi$ with $\phi(x_0) \neq 0$
 such that $\phi f \in C^\infty.$
 
	On the one hand, from the example of the Weierstrass function in 
	\eqref{eq:signatureWeierstrass} we see that in general
	\begin{align}\label{eq:signature_supp_notin_singsupp}
		\singsupp f \not\subseteq \supp \sigma f.
	\end{align}

On the other hand,
 let $f = e^{-\gamma x^2}$ be a Gaussian function with $\gamma >0,$ and let $\kappa$ be any signature wavelet. The singular support of $f$ is empty because $f$ is smooth. 
  Since  
  	\[
		\ps{f, \kappa_{a,0}} = \ps{\hat f, (\kappa_{a,0})^\vee} = \sqrt{\pi} \int_\R  e^{-\frac{\omega^2}{4\gamma }} (\kappa_{a,0})^\vee(\omega) \de \omega > 0, \qquad\text{for all }a > 0,
	\] 
the signature is equal to $1$ at $b = 0.$ 	Thus, in general,
	\begin{align}
		 \supp \sigma f  \not \subseteq \singsupp f.
	\end{align}
	The example of the Gaussian is particularly interesting: 
	even though the function can be approximated arbitrarily close by polynomials around the origin,
	its signature at the origin does not vanish.

Now, let us compare the signature to the local Sobolev regularity.
Recall that the local Sobolev regularity index $s_f(b)$ of $f$ at $b$
is defined by
\begin{equation} \label{eq:local_Sobolev_regularity}
  s_f(b) = \sup \left\{ s \in \R \st  \exists \; \varphi  \in \mathscr{D}(\R), \,\varphi(b) \neq 0, \mbox{ so that }  \varphi f \in W^{s,2}(\R)\right\}.
\end{equation}
(See, for example, \cite[Chapter 18.1]{hoermander2007analysis}.)
The local Sobolev regularity index is often (implicitly) used in signal analysis
since it is characterized by the decay of the moduli of wavelet coefficients,
cf. \cite{mallat2009wavelet}.
Whereas the local Sobolev regularity index is by definition a measure for the local order of differentiability, the interpretation of signature is not so obvious.
First notice 
a fundamental difference between Sobolev regularity and signature: the fractional Hilbert transform leaves the local Sobolev regularity index $s_f$ invariant, 
that is,
\begin{equation}
\label{eq:s_hilb}
   s_{\hilbert^{\alpha} f}(b) = s_f(b). 
\end{equation}
while this is not the case for   the signature
(see \autoref{thm:frac_hilbert_and_signature}).
We  argue that signature can be considered as an indicator of local symmetry at a singular point.
To this end, let us recall that the signature is imaginary at  step discontinuities and real at symmetric cusps of power type. This is a first indication that imaginary signatures correspond to locally antisymmetric features and that real signatures indicate locally symmetric features;  see
\autoref{tab:signature_symmetry}. 
A more rigorous argument can be derived from the change of the signature under the Hilbert transform.
The ordinary Hilbert transform  converts the symmetric part of a signal to an antisymmetric function, and vice-versa.
\autoref{thm:frac_hilbert_and_signature} states that the application of the  Hilbert transform to a signal induces a multiplication of its signature by $e^{\frac12 i \pi} = i.$ Thus, the change of symmetry is directly reflected in the rotation of the signature by $\frac{\pi}2$ in the complex plane.
We may explicitly observe this effect for the sign function $x\mapsto \sign x$  and its 
Hilbert transform (in a weak sense), the logarithm $x \mapsto 2\log |x|.$
The antisymmetric sign function has a signature of $i$ at the origin
whereas the signature of the logarithm equals $i \cdot i = -1.$

			 \def\thisfigscale{0.1}
			 
 \newcommand{\featurePlot}[2]{
 \tikzsetnextfilename{#2}
 \scalebox{\thisfigscale}{
 \begin{tikzpicture}
	\pgfplotsset{ticks=none}
    \begin{axis}[
      footnotesize,
      axis lines=none
]
\addplot[no marks]{#1};
    \end{axis}
  \end{tikzpicture}
  }
  }
			 
 \begin{table}
\centering
	\begin{tabular}{lll}
		\toprule
 		Signal feature  
&		Local symmetry & Signature \\
			\toprule
 	Step to right   \featurePlot{abs(x)/x}{stepRight} & antisymmetric & $+i$ \\
	Step to left \featurePlot{-abs(x)/x}{stepLeft} & antisymmetric & $-i$ \\
	\midrule 	
	Cusp upwards \featurePlot{-sqrt(abs(x))}{cuspUp} & symmetric & $+1$ \\
	Cusp downwards \featurePlot{sqrt(abs(x))}{cuspDown} & symmetric & $-1$ \\
\bottomrule 
	\end{tabular}
	\caption{Signatures of  frequently occurring feature types. 
	The signature is imaginary at the antisymmetric and real at the symmetric features.
	}
	\label{tab:signature_symmetry}
\end{table}

\section{Discretization and numerical experiments} \label{sec:numeric}

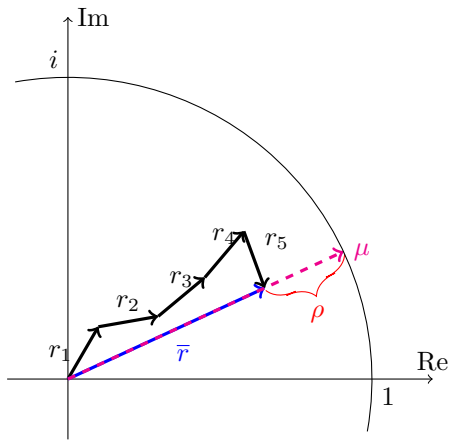
\begin{figure}[t]
  \centering
  \tikzsetnextfilename{directionalStat}
  \begin{tikzpicture}[x=4cm,y=4cm]
    \draw[->] (-.2,0) -- (1.2,0) node[above] {$\real$};
    \draw[->] (0,-.2) -- (0,1.2) node[right] {$\imag$};
    \tikzstyle{vectorend}=[coordinate,fill];
    \pgfmathparse{1 / sqrt(2)};
    \let\rad\pgfmathresult;
    \pgfmathparse{1 / 5};
    \let\step\pgfmathresult;
    \draw[vector] (0,0) -- ++ (60:\step) node[pos=0.5, left] {$r_{1}$} node[vectorend] (a){};
    \draw[vector] (a) -- ++ (10:\step) node[pos=0.5, above] {$r_{2}$} node[vectorend] (b){};
    \draw[vector] (b) -- ++ (40:\step) node[pos=0.5, above] {$r_{3}$} node[vectorend] (c){};
    \draw[vector] (c) -- ++ (50:\step)  node[pos=0.5, above] {$r_{4}$} node[vectorend] (d){};
    \draw[vector] (d) -- ++ (-70:\step) node[pos=0.5, above right] {$r_{5}$} node[vectorend] (e){};
    \def\myang{25};
    \draw[decorate,decoration={brace, mirror, amplitude=10pt, raise=1pt}, red] (e) -- (\myang:1) node[pos=0.5,below, yshift=-10pt, xshift=5pt] {$\rho$};
    \draw[sign, blue] (0,0) -- (e) node[pos=0.5,below right] {$\overline{r}$};
    \draw[sign, dashed] (0,0) -- (\myang:1) node [right]{$\mu $};
    \draw [](-10:1) arc (-10:100:1);
    \draw (1,0)  node[below right] {$1$};
    \draw (0,1)  node[above left] {$i$};
  \end{tikzpicture}
  \caption{Illustration of the moments of the directional statistics of the wavelet signs $\sign c_{j,b}$ for a fixed location $b$
  and five scale samples $a_j,$ where $j = 1,...,5.$
  The vectors $r_j = \frac{1}{5}\sign c_{j,b}$ are composed head to tail in the complex plane.
  }
  \label{Fig:directional_statistics}
\end{figure}

\begin{figure}[t]
  \centering
  \def\thisfigwidth{0.35\textwidth}
  \def\thisfigheight{0.25\textwidth}
    \def\dataFolder{experiments/signature1Dnew/}
  \def\noise{Noise0}
  \def\signalA{heaviside}
  \def\signalB{abs}
  \tikzstyle{signature}=[only marks, mark = square*, orange]
\tikzstyle{signaturevar}=[mark=none, line width=1pt, red]
\begin{tikzpicture}
    \begin{axis}[
      footnotesize,
      width=\thisfigwidth,
      height=\thisfigheight,
      axis y line=center,
      axis x line=center,
      xlabel=$x$,
      ylabel=$f(x)$,
      ymax=1.2,
      xtick={0,256,...,1024},
      ytick={0,0.5,...,1},
      legend style={at={(0.03,0.97)},anchor=north west},
      tick label style={font=\footnotesize},
]
            \addplot [plot] table[x =x,y =signal]
      {\dataFolder \signalA/\noise/PS\signalA Raw.table};  
    \end{axis}
\end{tikzpicture}
\hfill
 \begin{tikzpicture}
    \begin{axis}[
      footnotesize,
      width=\thisfigwidth,
      height=\thisfigheight,
      axis y line=center,
      axis x line=center,
      xlabel=$b$,
      ylabel=$|{r}(b)|$,
      ymin=0,
      xtick={0,256,...,1024},
      grid=major,
      legend style={at={(0.03,0.97)},anchor=north west},
      tick label style={font=\footnotesize},  
]
            \addplot [signaturevar] table[x =x,y =signatureAbs]
      {\dataFolder\signalA/\noise/PS\signalA Raw.table};  
    \end{axis}
    \end{tikzpicture}
\hfill
\begin{tikzpicture}
    \begin{axis}[
      footnotesize,
      width=\thisfigwidth,
      height=\thisfigheight,
      axis y line=center,
      axis x line=center,
      xlabel=$b$,
      ylabel=$\arg (\overline{\sigma} f(b))$,
      ymax=3.14,
      ymin=-3.14,
      ytick={-3.14 , -1.57, 0, 1.57, 3.14},
            yticklabels={$-\pi$ , $-\frac{\pi}2$, $0$, $\frac\pi2$, $\pi$},
       xmin= 1,
      xmax = 1024,
      xtick={0,256,...,1024},
            grid=major,
      legend style={at={(0.03,0.97)},anchor=north west},
      tick label style={font=\footnotesize}  
]
            \addplot [signature] table[x index = 0,y index = 1]
      {\dataFolder\signalA/\noise/PS\signalA ThreshAngles.table};  
    \end{axis}
    \end{tikzpicture}
\\[2ex]
\begin{tikzpicture}
    \begin{axis}[
      footnotesize,
      width=\thisfigwidth,
      height=\thisfigheight,
      axis y line=center,
      axis x line=center,
      xlabel=$x$,
      ylabel=$f(x)$,
      ymax=1.2,
      xtick={0,256,...,1024},
      ytick={0,0.5,...,1},
      legend style={at={(0.03,0.97)},anchor=north west},
      tick label style={font=\footnotesize},    
]
            \addplot [plot] table[x =x,y =signal]
      {\dataFolder\signalB/\noise/PS\signalB Raw.table};  
    \end{axis}
\end{tikzpicture}
    \hfill
\begin{tikzpicture}
    \begin{axis}[
      footnotesize,
      width=\thisfigwidth,
      height=\thisfigheight,
      axis y line=center,
      axis x line=center,
      xlabel=$b$,
      ylabel=$|{r}(b)|$,
      ymin=0,
      xtick={0,256,...,1024},
            grid=major,
      legend style={at={(0.03,0.97)},anchor=north west},
      tick label style={font=\footnotesize}  
]
            \addplot [signaturevar] table[x =x,y =signatureAbs]
      {\dataFolder\signalB/\noise/PS\signalB Raw.table};  
    \end{axis}
\end{tikzpicture}
    \hfill
\begin{tikzpicture}
    \begin{axis}[
      footnotesize,
      width=\thisfigwidth,
      height=\thisfigheight,
      axis y line=center,
      axis x line=center,
      xlabel=$b$,
      ylabel=$\arg(\overline{\sigma} f(b))$,
      ymax=3.15,
      ymin=-3.15,
      xmin= 1,
      xmax = 1024,
      xtick={0,256,...,1024},
      ytick={-3.14 , -1.57, 0, 1.57, 3.14},
      yticklabels={$-\pi$ , $-\frac{\pi}2$, $0$, $\frac\pi2$, $\pi$},
            grid=major,
      legend style={at={(0.03,0.97)},anchor=north west},
      tick label style={font=\footnotesize}  
]
            \addplot [signature] table[x index = 0,y index = 1]
      {\dataFolder\signalB/\noise/PS\signalB ThreshAngles.table};  
    \end{axis}
\end{tikzpicture}
  \caption{The discrete signature of a jump and a cusp.
  The absolute value of the resultant vector $\overline{r}_b$
  is close to one at the singularities
  and much lower at the other points (second column).
  In the third column, the discrete signature according to 
  \eqref{eq:discrete_signature}
  is depicted as a phase angle. The signature is close to  the angles $\pm\frac{\pi}2,$
  for the step  and around $\pm\pi$ for the cusp.
 	The extra points at the boundaries are due to the periodization of the signal.
  }  
    \label{Fig:signature_simple}
\end{figure}
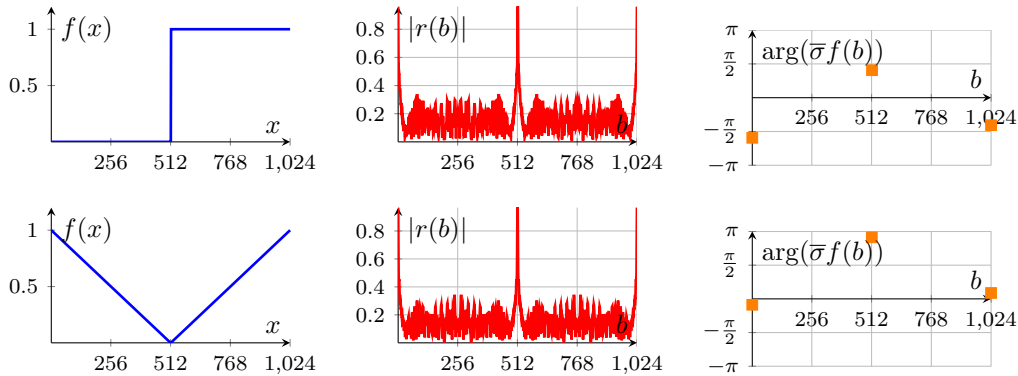

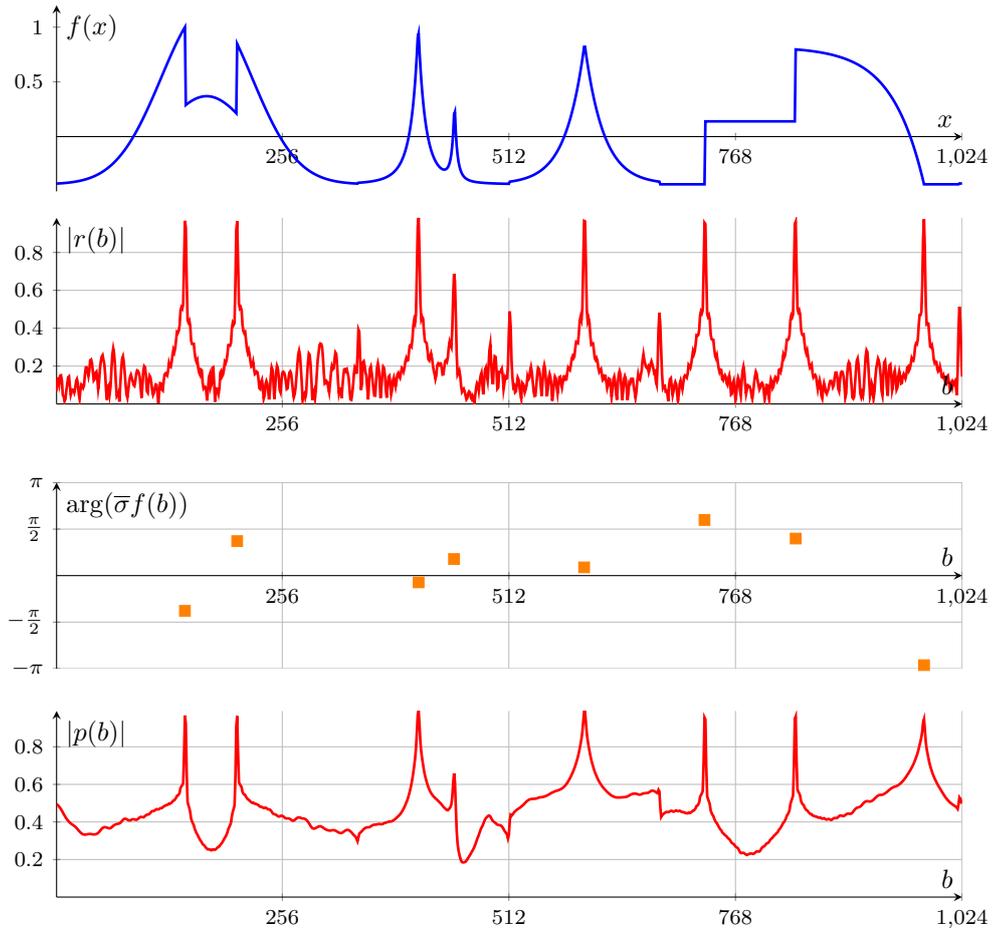
\begin{figure}[]
  \flushright
  \def\thisfigwidth{1\textwidth}
  \def\thisfigheight{0.3\textwidth}
  \def\dataFolder{experiments/signature1Dnew/}
  \def\noise{Noise0}
  \def\signalA{piecewiseRegular}
  \tikzstyle{signature}=[only marks, mark = square*, orange]
\tikzstyle{signaturevar}=[mark=none, line width=1pt, red]
    \begin{tikzpicture}
    \begin{axis}[
      footnotesize,
      width=\thisfigwidth,
      height=\thisfigheight,
      axis y line=center,
      axis x line=center,
      xlabel=$x$,
      ylabel=$f(x)$,
      ymax=1.2,
      ymin= -0.5,
      xtick={0,256,...,1024},
      ytick={0,0.5,...,1},
      legend style={at={(0.03,0.97)},anchor=north west},
      tick label style={font=\footnotesize},
]
            \addplot [plot] table[x =x,y =signal]
      {\dataFolder \signalA/\noise/PS\signalA Raw.table};  
    \end{axis}
    \end{tikzpicture}
     \\[2ex]
  \begin{tikzpicture}
    \begin{axis}[
      footnotesize,
      width=\thisfigwidth,
      height=\thisfigheight,
      axis y line=center,
      axis x line=center,
      xlabel=$b$,
      ylabel=$|r(b)|$,
      ymin=0,
      xtick={0,256,...,1024},
      grid=major,
      legend style={at={(0.03,0.97)},anchor=north west},
      tick label style={font=\footnotesize},  
]
            \addplot [signaturevar] table[x =x,y =signatureAbs]
      {\dataFolder\signalA/\noise/PS\signalA Raw.table};  
    \end{axis}
    \end{tikzpicture}
    \\[2ex]
    \begin{tikzpicture}
    \begin{axis}[
      footnotesize,
      width=\thisfigwidth,
      height=\thisfigheight,
      axis y line=center,
      axis x line=center,
      xlabel=$b$,
      ylabel=$\arg (\overline{\sigma} f(b))$,
      ymax=3.14,
      ymin=-3.14,
      ytick={-3.14 , -1.57, 0, 1.57, 3.14},
            yticklabels={$-\pi$ , $-\frac{\pi}2$, $0$, $\frac\pi2$, $\pi$},
       xmin= 1,
      xmax = 1024,
      xtick={0,256,...,1024},
            grid=major,
      legend style={at={(0.03,0.97)},anchor=north west},
      tick label style={font=\footnotesize}  
]
            \addplot [signature] table[x index = 0,y index = 1]
      {\dataFolder\signalA/\noise/PS\signalA ThreshAngles.table};  
    \end{axis}
    \end{tikzpicture}
     \\[2ex]
\begin{tikzpicture}
    \begin{axis}[
      footnotesize,
      width=\thisfigwidth,
      height=\thisfigheight,
      axis y line=center,
      axis x line=center,
      xlabel=$b$,
      ylabel=$|p(b)|$,
      ymin=0,
      xtick={0,256,...,1024},
      grid=major,
      legend style={at={(0.03,0.97)},anchor=north west},
      tick label style={font=\footnotesize}
]
            \addplot [signaturevar] table[x =x,y =phaseCong]
       {\dataFolder\signalA/\noise/PS\signalA Raw.table};   
    \end{axis}
    \end{tikzpicture}    
         \caption{The discrete signature of a sample signal (top) taken from Wavelab \cite{donoho2006wavelab}.
  We observe that the absolute value of the mean $|r(b)|$
  is large at the feature points
  and much lower at the other points (second row).
  The third plot depicts the discrete signature in phase angle representation.
  We see that the signature is near the angles $\pm\frac{\pi}2$
  at the step-like points, and around $\pi$ and $0$ at cusp-like points.
   Phase congruency is shown for comparison (fourth row).
  }  
    \label{Fig:signature_piecewise}
\end{figure}

In practice, only a finite number of wavelet scales $\{a_j\}_{j=1}^N$ is available,
so we have to estimate the limiting behavior from a finite number of samples
\[
	c_{j,b} =  \ps{f, \kappa_{a_j,b}}, \quad j = 1,...,N.
\]
A standard method for estimating local Sobolev regularity 
is to determine the decay behavior of the wavelet coefficients 
by linear regression on the magnitudes of the available coefficients.
In our case it seems reasonable to measure the variation 
around a mean direction.
To this end we regard the signs of the wavelet coefficients
as directions within the complex plane
and then utilize the framework of {directional statistics}.
Directional statistics describes the statistical properties of a sample set of directions;
see e.g. \cite{fisher1996statistical} for an introduction to this topic
and \cite{ehler2011frame} for applications in frame theory.

An elementary descriptor of the data set $\{  \sign c_{j,b} \}_{j=1}^N$
is the mean value of the sample set, given by
\begin{equation}\label{eq:mean_resultant_vector_1d}
  {r}(b) =  \frac 1 N \sum_{j= 1}^N  \sign c_{j,b} = \frac 1 N { \sum_{c_{j,b} \neq 0} \frac{   {c_{j,b}} }{ \abss{c_{j,b}}} }.
\end{equation}
The value ${r}(b)$ is called {mean resultant vector.}
The moments of the data set, the {directional mean} $\mu(b) \in \C$ and the {directional variance} $\rho(b) \in [0,1],$
can be directly derived from the mean resultant vector.
The directional mean is defined by
\begin{equation*}
  \mu(b) = \sign{ {r}(b)}
\end{equation*}
and the directional variance by
\begin{equation*}
  \rho(b) = 1 - \abs{{r}(b)}.
\end{equation*}
These quantities  are illustrated in \autoref{Fig:directional_statistics}.
We consider the signature to be non-zero 
if their directional variance is low;
more precisely, if the directional variance $\rho(b)$ falls below some threshold $\tau \in [0,1].$
 Hence, we propose a discrete estimate  $\overline{\sigma} f (b)$ of the signature of the form
 \begin{align}\label{eq:discrete_signature}
 	\overline{\sigma} f (b) 
	= 
	\begin{cases}
		\sign{ {r}(b) }, &\text{if }|{r}(b)| \geq 
		\tau, \\
		0, &\text{else}.
	\end{cases}
 \end{align}
We obtain the following elementary consistency result for our discretization.
\begin{proposition}\label{lem:cesaro_limit}
Let $f$ be a tempered distribution,  $\kappa$ be a signature wavelet, 
and $\{a_j\}_{j \in \N}$ a sequence such that $a_j \to 0,$ 
If  $\sigma f(b) \neq 0$  for some  $b \in \R$
then
\begin{align}\label{eq:cesaro_limit}
	\lim_{N \to \infty} \frac{1}{N} \sum_{j=1}^N\sign\ps{f, \kappa_{a_j,b}} = \sigma f(b).
\end{align}
\end{proposition}
\begin{proof}
Let $\kappa$ be a signature wavelet.
As  $\sigma f (b)$ is non-zero, the sequence $ \{ \sign
  {\ps{f, \complexwvlt_{a_j,b}}} \}_{j\in \N}$ is convergent.
We observe that the limit in \eqref{eq:cesaro_limit} is the Cesàro sum of this sequence.
The claim follows from the fact that, for a convergent sequence,  
its Cesàro sequence  converges to the same limit,
cf.  \cite[Chapter 5.7]{hardy1949divergent}.
\end{proof}

Next, we derive a  value for the threshold $\tau.$
Assume that the signs of the wavelet coefficients $\{\sign c_{j,b}\}_{j=1}^N$ at some location $b$ 
are independent and follow a uniform distribution on the circle.
Then the signature at $b$ should be equal to zero
 with high probability.
Therefore, we choose $\tau$ such that the probability  that the resultant vector exceeds 
the threshold, $P\p{ \abss{r(b)} \geq \tau}$, is smaller than some significance level $\alpha,$ say $\alpha = 0.05.$
To derive a lower bound on $\tau,$ 
recall that a complex number $x + iy$ of modulus one can be represented as a $(2 \times 2)$ matrix of the form
$(\begin{smallmatrix} x & -y\\ y & x  \end{smallmatrix})$
whose spectral norm is equal to one.
This allows to apply the matrix Bernstein inequality
(see Theorem 1.6 of \cite{tropp2012user}) which yields
\[\textstyle
	P\p{|r(b)|  \geq \tau} = 	P\p{|\sum_{j= 1}^N  \sign c_{j,b}|  \geq N \tau} \leq 4 \exp\p{-\frac{\frac{1}{2}(N\tau)^2}{N + \frac{N\tau}{3}}} 
	= 4 \exp\p{-\frac{\frac{1}{2}N\tau^2}{1 + \frac{\tau}{3}}} = 
	\alpha.
\]
Solving for $\tau$ yields
\begin{equation}\label{eq:tau}
	\tau = \frac{1}{N} \p{-\frac{1}{3}\log \p{\frac\alpha4} + \sqrt{ {\frac19{\log^2\p{\frac\alpha4}}}
	- 2 N \log\p{\frac\alpha4}   }}.
\end{equation}

 For all experiments, we used the Meyer-type signature wavelet \eqref{eq:meyer_signature_wavelet}. 
In order to get a sufficient amount of  samples we
use five octaves and five scales per octave, which corresponds to  a scale sampling of the form 
 $a_j = 2^{-\frac{j}{5}},$ with $j = 0, \ldots,24.$ 
We take integer samples of $f$ and we sample $b$ on the same grid for each scale.
In the experiments, 
it can be often observed that  two adjacent points exceed the threshold.
Since they typically correspond to the same feature, 
we set $r(b)$ to zero if $|r(b)|$ is not the maximum of the three values
$|r(b-1)|,$ $|r(b)|,$ and $|r(b+1)|.$ 
This technique is called non-maximum suppression; it is frequently used in feature detection pipelines
  to get sharply localized feature points
 \cite{canny1986computational}.
For the significance level $\alpha = 0.05$ and $N=25$ we get 
from \eqref{eq:tau} the threshold $\tau \approx 0.65.$ 

In \autoref{Fig:signature_simple}, we see the discrete signatures of a pure jump and a pure cusp singularity.
We observe that the mean resultant vector exceeds the threshold at the singular points
and the orientation in the complex plane is given by $i$ and $-1,$ respectively.
This is in agreement with the theoretical results for pure cusps and pure steps as given in 
\autoref{tab:signature_symmetry}. 
We can observe similar results for a more complicated signal \autoref{Fig:signature_piecewise}.

We eventually compare our discretization to phase congruency \cite{kovesi1999image}.
The quantity of interest in   phase congruency 
is the modulus of 
\begin{align}\label{eq:phaseCong}
   p(b) = \frac{  \sum_{j = 1}^N c_{j,b} }{ \sum_{j = 1}^N \abss{c_{j,b} } + \epsilon}.
\end{align}
with some $\epsilon >0.$
In contrast to the discretization of the signature \eqref{eq:discrete_signature},
the normalization is applied after summation of the coefficients. 
This comes with the following issue, termed frequency spread in \cite{kovesi1999image}.
If one coefficient is much larger than the others, then $|{p}(b)|$ is large, 
even if the wavelet coefficients $c_{j,b}$ point into different directions.
This effect is illustrated in \autoref{Fig:signature_piecewise}.
In \cite{kovesi1999image}, this effect is attenuated 
by introducing a sigmoidal weight function. However, 
this comes with the cost of two extra empirical parameters.
The discretization of the signature \eqref{eq:discrete_signature} 
does not encounter this problem because the normalization is applied before summation
so that every coefficient is weighted equally.

\section{Summary and outlook}

We have proposed a new tool for signal analysis that is based on complex wavelet signs.
Signature is capable of delineating and classifying isolated salient feature points of a signal.
 We proved that the signature is imaginary  at a jump discontinuity, whereas  it is real at a  symmetric cusp of power type. Furthermore, it rotates in the complex plane under fractional Hilbert transforms. 
In addition, we showed that singularities regarded in the classical sense need not coincide with singularities described in terms of signature. 
We have seen that signature classifies certain random signals as non-salient. 
Finally, we have proposed a discretization based on directional statistics
which is consistent with the theoretical concepts developed in this paper.

Although the signature depends in general on the employed wavelet
almost all assertions made in this paper are independent of this particular choice.
It is an open question 
under which conditions signature is independent of the wavelet.
We conjecture that  signature is invariant under 
 fractional Laplacians.
A further direction of research is the multidimensional extension of signature.

\section*{Acknowledgement}
This work has received funding from the German Federal Ministry for Education and Research under SysTec Grant 0315508 and
from the European Research Council under the European Union's Seventh Framework Programme (FP7/2007-2013) / ERC grant agreement no.~267439.
It was partially supported by the research cooperation \enquote{Conformal monogenic frames for image analysis}  funded by the program ‚Acções Integradas Luso-Alemãs/DAAD’. This bilateral program is funded by the  
Deutsche Akademischer Austauschdienst DAAD by means of the German Federal Ministry of Education and Research and of the Portuguese Ministry for Science, Education and Technology.
We would like to thank the reviewers for their valuable comments and suggestions that helped to improve the quality of the paper.

\bibliographystyle{plain}
\bibliography{Phase}

\end{document}